\documentclass[11pt,draftcls,onecolumn]{IEEEtran}
\usepackage{amsfonts}
\usepackage{amsmath,amssymb}
\usepackage{graphicx}
\usepackage{graphics}
\usepackage{subfigure}
\usepackage{multirow}
\usepackage{caption}
\usepackage{makecell}
\usepackage{cite}
\usepackage{color}
\usepackage{cases}

\newtheorem{theorem}    {Theorem}
\newtheorem{definition} {Definition}
\newtheorem{corollary}  {Corollary}
\newtheorem{lemma}      {Lemma}

\newtheorem{remark}     {Remark}

\begin{document}

\title{Convergence Rate of Accelerated Average Consensus with Local
Node Memory: Optimization and Analytic Solutions \thanks{%
This research was supported by the National Science Foundation of China
(grant 61625305 and 61701355). }}
\author{Jing-Wen Yi, Li Chai*, and Jingxin Zhang
\thanks{%
* Corresponding author. Li Chai is with the Engineering Research Center of
Metallurgical Automation and Measurement Technology, Wuhan University of
Science and Technology, Wuhan, China. e-mail: chaili@wust.edu.cn (L. Chai).}%
\thanks{%
Jing-Wen Yi is with the same University. e-mail: yijingwen@wust.edu.cn (J.W.
Yi).} \thanks{%
Jingxin Zhang is with the School of Science, Computing and Engineering Technology,
Swinburne University of Technology, Melbourne, Australia. e-mail:
jingxinzhang@swin.edu.au (J. Zhang)}}
\maketitle

\begin{abstract}
Previous researches have shown that adding local memory can accelerate the consensus. It is natural to ask questions like what is the fastest rate 
achievable by the $M$-tap memory acceleration, and what are the corresponding control parameters. This paper introduces
a set of effective and previously unused techniques to analyze the convergence rate of accelerated consensus with $M$-tap memory of local nodes and to design the control protocols. These effective techniques, including the Kharitonov stability theorem, the Routh stability criterion and the robust stability margin, have led to the following new results: 1) the direct link between the convergence rate and the control parameters; 2) explicit formulas of the optimal convergence rate and the corresponding optimal control parameters for $M \leq 2$ on
a given graph; 3) the optimal worst-case convergence rate and the corresponding optimal control parameters for the memory $M \geq 1$ on a set of uncertain graphs. We show that the acceleration with the memory $M=1$ provides the optimal convergence rate in the sense of worst-case performance. Several numerical examples are given to demonstrate the validity and performance of the theoretical results.
\end{abstract}

\IEEEpeerreviewmaketitle

\section{Introduction}


Over the past decade, distributed consensus algorithms have received renewed
interests due to their wide applications in various fields, such as social
networks, wireless sensor networks, and smart grids, to mention a few \cite%
{[SFM07], [CYRC13], [KCM16]}. For distributed average consensus of
multi-agent system (MAS), each agent only gets information from its local
neighbours, and the whole network of agents coordinates to reach the average
of their initial states.

The distributed average consensus can be achieved in many ways. A well
studied algorithm is the linear iterative update of the form\cite{[SFM07]}
\begin{equation}
x_{i}(k+1)=x_{i}(k)+\varepsilon \sum\limits_{j\in \mathcal{N}%
_{i}}a_{ij}(x_{j}(k)-x_{i}(k)),  \label{eq1}
\end{equation}%
where $x_{i}(k)$ is the $i$th agent state at time $k$, $\mathcal{N}_{i}$ is
the set of neighbors of agent $i$, and $\varepsilon >0$ is the control gain,
$a_{ij}\geq 0$ is the coupling weight between agent $i$ and agent $j$. Let $%
x(k)=[x_{1}(k),\cdots ,x_{N}(k)]^{T}\in \mathbb{R}^{N}$, the overall system
can be written as
\begin{equation}
x(k+1)=(I-\varepsilon \mathcal{L})x(k)=:Wx(k).
\end{equation}%
The average consensus is said to be achieved if
$\lim_{k\rightarrow \infty }{x_{i}(k)}=\frac{1}{N}\sum_{j=1}^{N}{%
x_{j}(0)}=:\bar{x}$ 
for any initial state $x(0)$.
It is clear that the system (2) can reach
average consensus if and only if the weight matrix $W$ has a simple
eigenvalue $1$ and all the other eigenvalues are within the unit circle.
The convergence rate of (2) is determined by the second largest eigenvalue
modulus of the weight matrix $W$ \cite{[OT09]}. Understanding the
relationship between the convergence rate and the property of the network is
a key question to design accelerated consensus algorithms. Intuiatively, the
better the connectivity of the network, the faster the algorithms might be.

To accelerate the convergence rate, optimization design of the weight matrix
$W$ or the coupling weights $a_{ij}$ have been proposed in \cite{[XB04]}. Xiao,
Boyd, and their collaborators \cite{[XBK07]} proposed computational methods
to design the optimal weight matrix and solve the fastest distributed linear
averaging and least-mean-square consensus problems.
Jakovetic et al. \cite{[JXM10]} extended the weight optimization problem to
spatially correlated random topologies by adopting the consensus mean-square
error convergence rate or the mean-square deviation rate as the optimization
criterion. Erseghe et al. \cite{[EZD11]} provided an effective indication on
how to choose the weight matrix for optimized performance by an ADMM-based
consensus algorithm. Kokiopoulou and Frossard \cite{[KF09]} applied a
polynomial filter on the weight matrix to shape its spectrum in order to
increase the convergence rate, and used a semidefinite program to optimize
the polynomial coefficients. Montijano et al. \cite{[MMS13]} proposed a fast
and stable distributed algorithm based on Chebyshev polynomials to solve the
weight optimization problem and accelerate consensus convergence rate.

For MASs on time-varying graphs, the convergence rate is considered by using
a variety of consensus protocols \cite{[TBA86],[JLM03],[AB09],[CMA08],[NOOT09],[NL17]}.
Cao et al. \cite{[CMA08]} adopted the theory of nonhomogeneous
Markov chain to accelerate consensus on time-varying graphs, and derived the
worst case convergence rate of $1-O(\frac{1}{N^{N-1}})$. Nedic et al. \cite%
{[NOOT09]} investigated a large class of averaging algorithms over time
varying topologies, and provided an algorithm with the best convergence rate
of $1-O(\frac{1}{N^{2}})$. Nedic and Liu \cite{[NL17]} presented two novel
convergence rate results for unconstrained and constrained consensus over
time-varying graphs, and established the convergence rate with the exponent
of the order $1-O(\frac{1}{N{log_{2}}N})$ for special tree-like regular
graphs.

For graphs with fixed coupling weights, there are generally two ways to
accelerate the convergence rate. One is using time-varying
control gains \cite{[KG09],[HJOV14],[HSJ15],[K14],[SK15],[SKM14],[YCZ20]} and
the other is using local memory \cite{[MGS98],[CSY06],[JJ08],[GJS11],[LACM13],
[LM11],[OCR10],[AOC09],[O17],[PDK20]}. By choosing control gains $%
\varepsilon (k)$ equal to the reciprocal of nonzero Laplacian eigenvalues,
finite-time consensus can be achieved. This result has been obtained by
different methods, including matrix factorization \cite{[KG09],[HJOV14],[HSJ15]},
minimal polynomial \cite{[K14],[SK15]} and graph filtering \cite%
{[SKM14],[YCZ20]}. However, it is usually difficult to
compute the exact eigenvalues of Laplacian matrix in practice,
especially for a large and complex network. Moreover, the eigenvalues
are global information of the network. By establishing the direct link
between the consensus convergence rate and the graph filter corresponding to
the consensus protocol, \cite{[YCZ20]} converted the fast consensus problem
to the design of a polynomial graph filter. If the nonzero eigenvalues of the
Laplacian matrix are within $[\lambda _{2},\lambda _{N}]$, \cite{[YCZ20]}
provided the explicit formulas for the convergence rate and the
corresponding parameters of the protocols, and showed that the limitation
of the optimal convergence rate is $\frac{\sqrt{\lambda _{N}/\lambda _{2}}-1}
{\sqrt{\lambda _{N}/\lambda _{2}}+1}$ for periodically time-varying control
sequences.

By adding local memory of agents to the control protocol, the convergence
can also be accelerated. Muthukrishnan et al. \cite{[MGS98]} proposed the
augmented state-update equation by incorporating one-tap memory of the form
\begin{equation}
x(k+1)=(1-\alpha )Wx(k)+\alpha x(k-1),
\end{equation}%
where $\alpha \in (-1,1)$ was a real parameter dependent on $W$. %
By choosing $\alpha \in (-(\rho _{s}(W))^{2},0)$, with $\rho _{s}(W)$ the second largest eigenvalue modulus of $W$, Liu and Morse \cite{[LM11]} proved that
the system (3) had a faster convergence rate than that of the original system (2), and the fastest convergence rate $\frac{\rho _{s}(W)}{1+\sqrt{1-(\rho _{s}(W))^{2}%
}}$ was attained at $\alpha =\frac{1-\sqrt{1-(\rho
_{s}(W))^{2}}}{1+\sqrt{1-(\rho _{s}(W))^{2}}}$. Oreshkin et al \cite{[OCR10]}
considered a more general acceleration scheme with one-tap memory
\begin{equation}
x(k+1)=(1-\alpha +\alpha {\beta _{3}})Wx(k)+\alpha \beta _{2}x(k)+\alpha
\beta _{1}x(k-1),
\end{equation}%
where $\beta _{1},\beta _{2},\beta _{3}$ were prespecified real-valued
parameters. Under the spectrum assumption of $|\lambda _{N}(W)|\leq
|\lambda _{2}(W)|$, Oreshkin et al. \cite{[OCR10]} provided a feasible design
of the mixing parameter $\alpha $ by assuming $\beta _{1}+\beta _{2}+\beta
_{3}=1$ and $\beta _{2}\geq 0,\beta _{3}\geq 1$. Aysal et al. \cite{[AOC09]}
further proposed the state-update equations with multiple-tap memory, but
only considered the simplest case based on the current state without memory,
which coincides with the results of Xiao and Boyd \cite%
{[XB04]}. The schemes proposed in \cite{[LM11],[OCR10],[AOC09]} all
assumes that the agent can only use its own one-tap memory to accelerate
consensus. Olshevsky \cite{[O17]} further proposed a protocol that each
agent can use the neighbours' memory to improve the convergence speed.
The total number of nodes was needed to design the constant control gain.
Pasolini et al. \cite{[PDK20]} utilized neighbours' past information to design
a consensus algorithm consisting of an FIR loop filter, with each agent keeping in
memory the last $M$-taps received by its neighbours. It was claimed that the analytical
solution of this consensus algorithm is unavailable except in the special case $M=1$.

Despite the excellent works discussed above, there are still some fundamental
problems unsolved in the accelerated consensus with memory information.
To the best of our knowledge, there has been no report in the literature on the fastest convergence rate of accelerated consensus achievable by the $M$-tap memory of local nodes and the corresponding optimal parameters. Intuitively, it seems that the longer the memory or the larger the $M$ is, the faster the rate could be. However, this is not always true. As shown later in this paper, the optimal convergence rate does not change when the memory of (4) is increased by an additional term ${\beta_0}x(k-2)$. The complexity like this and the limitation of current analysis tools for the problem are the main reasons, we believe, for the insufficient results in the literature on the optimal convergence rate and optimal consensus protocols for the accelerated consensus with memory information.

The above facts have motivated us to introduce, in this paper, a set of effective and previously unused techniques for the analysis and design of accelerated consensus with $M$-tap memory of local nodes. These techniques include the Kharitonov stability theorem, the Routh stability criterion and the robust stability margin. Using these effective techniques we derive the following new results.

For MAS with $N$ agents and $M$-tap memory, we present a
key result (Theorem 1) that the convergence rate equals the maximum
modulus root of $N$ polynomials with order no greater than $M+1$.
By Kharitonov theorem, we show that the convergence rate equals the
maximum modulus root of only three polynomials for any $N$ when $M \leq 2$.
For MAS with $M=1$ and $M=2$ on fixed graphs, we derive the analytic formulas of the optimal convergence rate and the corresponding optimal parameters. Using these formulas we show that the optimal convergence rate for $M=2$ equals that of $M=1$, hence, adding one more memory does not accelerate the convergence rate.

Using an example MAS on a star graph, we show that the convergence rate in this special case can be accelerated by $M=3$, and point out that it may be interesting to investigate MASs with $M \geq 3$ on fixed graphs.
Since uncertain graphs are common in practice, we turn our attention to MASs on a set of uncertain graphs, instead of MASs with $M\geq 3$ on fixed graphs. We show that the acceleration with memory $M=1$ yields the optimal convergence rate for the worst-case performance, hence, one-tap memory provides the fastest convergence rate for MASs on a set of uncertain graphs or a graph with very large $N$.

The rest of this paper is organized as follows. Section II introduces some technical tools of robust stabilization to be used in the paper and formulates the problem of accelerated average consensus. Section III 
investigates the convergence rate optimization of accelerated average consensus to derive new analytical solutions for the optimal convergence rate and the corresponding control parameters. To deal with the
more general cases in practice, Section IV utilizes the gain margin optimization of robust stability to derive the optimal convergence rate for an uncertain graph set and the corresponding control parameters, and relates these to the results of Section III.
Section V provides numerical examples to demonstrate the validity and performance of the results. Finally, Section VI concludes the paper.


\section{Preliminaries and Problem Formulation}

In this section, we briefly review some basic results from the robust stability
theory to be used in later sections.
Then we present the consensus algorithm accelerated by node memory and formulate the problem of convergence rate optimization.

\subsection{Karitonov stability theorem and gain margin optimization}

Throughout this paper, the variables $s$ and $z$ in transfer functions or
polynomials indicate the continuous-time system and the discrete-time system, respectively. A continuous-time linear time invariant (LTI) system with transfer function $G(s)$ is said to be stable if all its poles are on the left half of the
(complex) $s$-plane with negative real parts. A discrete-time LTI system with transfer function $G(z)$ is said to be stable if all its poles are within the unit circle of the (complex) $z$-plane. Since the poles of a transfer function are the roots of its denominator polynomial, we can define the stability of polynomials.

\begin{definition}
A continuous-time polynomial $H(s)$ is said to be stable if all its roots
have negative real parts. Accordingly, a discrete-time polynomial $H(z)$ is
said to be stable if all its roots are within the unit circle.
\end{definition}

Consider the continuous-time interval polynomial $H(s)=%
\sum_{k=0}^{n}a_{k}s^{k},$ $a_{k}\in \lbrack \underline{a}_{k},\bar{a}_{k}].$
According to Kharitonov stability theorem \cite{[BZ86], [Jury90]}, a necessary and sufficient condition for the robust stability is that a particular $four$ of the $2^{n+1}$ corner polynomials are stable. However, this theorem does not hold for discrete-time systems.
In fact, even the stability of all $2^{n+1}$ corner polynomials is not sufficient to guarantee the stability of the
whole set of discrete-time interval polynomials.
Generally,  for the discrete-time interval polynomials with degree greater than three, their stability condition becomes very complex and there is not a counterpart of Kharitonov stability theorem. Nevertheless, for the degree less than or equal to three, the following Lemma holds \cite{[BZ86],[Cie87]}.



\begin{lemma}
Consider the interval polynomial $H(z)=\sum_{k=0}^{n}a_{k}z^{k},$ $a_{k}\in
\lbrack \underline{a}_{k},\bar{a}_{k}]$, $n \leq 3$. A necessary and sufficient
condition for its stability is that all $2^{n+1}$ corner polynomials are stable.
\end{lemma}

The $H_{\infty }$ norm of a stable system $G(s)$ is defined as $\left\Vert
G(s)\right\Vert _{\infty }:=\sup_{\omega }\left\vert G(j\omega )\right\vert.$
Let $P(s)$ be the nomimal plant and $P_{k}(s)=kP(s)$. The gain margin
optimization problem is to find the largest $\bar k$ such that there exists a
controller $C(s)$ achieving internal stability for every $P_{k}(s)$ with $%
1\leq k\leq \bar k.$ Denote the largest $\bar k$ as $k_{\sup }=\sup {\bar k}$ and
call it the optimal gain margin. $k_{\sup }$ can be computed directly from $%
\gamma _{\inf }:=\inf_{C(s)}\left\Vert T(s)\right\Vert _{\infty },$ where $T(s)=%
\frac{P(s)C(s)}{1+P(s)C(s)}.$ Lemma \ref{lemma2} below derived from \cite{[DFT92]} (Chapter 11) summarizes the computation method of $\gamma _{\inf }$ and $k_{\sup }$ for a given $P(s)$. 

\begin{lemma}
\label{lemma2}
Let $P(s)=\frac{U(s)}{V(s)}$ be a coprime factorization satisfying $
U(s)Y_{u}(s)+V(s)Y_{v}(s)=1$, where $U(s),V(s),Y_{u}(s)$ and $Y_{v}(s)$ are stable transfer functions. Let $c_{i},i=1,\cdots, n,$ be the zeros of $U(s)V(s)
$ in the right half of complex plane, and $b_{i}=U(c_{i})Y_{u}(c_{i}),i=1,\cdots ,n.$
Define $B_{1}=\left( \frac{1}{c_{i}+\bar{c}_{j}}\right)_{ij} ,B_{2}=\left( \frac{%
b_{i}\bar{b}_{j}}{c_{i}+\bar{c}_{j}}\right)_{ij}$, where $\bar{c}_j$ and $\bar{b}_j$
are the complex conjugate of ${c}_j$ and ${b}_j$, $i,j=1,\cdots,n$.
Then the following statements hold:\newline
(i) $\gamma _{\inf }$ equals the square root of the largest eigenvalue of
$B_{1}^{-1}B_{2},$ that is, $\gamma _{\inf }=\sqrt{\overline{\lambda }%
(B_{1}^{-1}B_{2})}.$\newline
\ (ii) If $P(s)$ is stable or minimum phase, then $k_{\sup }=\infty.$
Otherwise, $k_{\sup }=\left( \frac{\gamma _{\inf }+1}{\gamma _{\inf }-1}%
\right) ^{2}.$
\end{lemma}

\subsection{Accelerated average algorithms}

Consider a set of $N$ agents communicating information through a network
described by an undirected graph $\mathcal{G}=(\mathcal{V},\mathcal{E},%
\mathcal{A}),$ where $\mathcal{V}=\{{{\nu }_{1}},{{\nu }_{2}},\cdots ,{{\nu }%
_{N}}\}$ is the set of vertices, $\mathcal{E}\subseteq \mathcal{V}\times
\mathcal{V}$ is the set of edges, and $\mathcal{A=}\left[ a_{ij}\right]$ is
the adjacency matrix, with $a_{ii}=0$ and ${a_{ij}}={a_{ji}}>0$ if and only if $({{\nu }_{i}},{{\nu }_{j}})\in \mathcal{E}$.

A multi-agent system on $\mathcal{G}=(%
\mathcal{V},\mathcal{E},\mathcal{A})$ with $M$-tap memory is in the form
\begin{eqnarray}
{x_{i}}(k+1) &=&{x_{i}}(k)+{u_{i}}(k),  \label{agenti} \\
{u_{i}}(k) &=&{\alpha }\sum\limits_{j\in {\mathcal{N}_{i}}}{{a_{ij}}({x_{j}}%
(k)-{x_{i}}(k))}+\sum\limits_{m=0}^{M}{\theta {_{m}x_{i}}(k-m)},
\label{control}
\end{eqnarray}%
where ${x_{i}}(k)$ and ${u_{i}}(k)$ are respectively the state and the control signal of the $i$-th agent, $i=1,\cdots, N,$ and $\alpha ,{\theta _{0}},{%
\theta _{1}},\cdots ,{\theta _{M}}$ are the parameters to be designed. As
seen from (\ref{control}), each agent updates its state by using its own past
states stored in memory and its neighbors' current states. The initial
values of each agent are set as
\begin{equation}
x_{i}(-M)=\cdots =x_{i}(-1)=x_{i}(0),i=1,\cdots ,N.  \label{Initialvalue}
\end{equation}

\begin{definition}
The average consensus of the MAS (\ref{agenti})-(\ref{control}) is said to
be reached asymptotically if for any initial states $x_{i}(0),$ $%
\lim_{k\rightarrow \infty }x_{i}(k) =\frac{1}{N}\sum_{j=1}^{N}{%
x_{j}(0)} :=\bar{x},$ $i=1,\cdots, N.$
\end{definition}

The degree of vertex ${\nu }_{i}$ are represented by $d_{i}=\sum%
\nolimits_{j=1}^{N}{{a_{ij}}}$. The Laplacian matrix of $\mathcal{G}$ is
defined as $\mathcal{L}=\mathcal{D}-\mathcal{A}$, where $\mathcal{D}%
:=diag\{d_{1},\cdots ,d_{N}\}$ is the degree matrix. It is obvious that the
Laplacian matrix $\mathcal{L}$ is positive semedifinite and all the
eigenvalues are real and can be written in ascending order as $0={\lambda
_{1}}\leq {\lambda _{2}}\leq \cdots \leq {\lambda _{N}}\leq 2\bar{d}$, where
$\bar{d}=\max_{i}\{{d_{i}}\}$ is the maximum degree of the graph.

Denote ${x(k)}=\left[ x_{1}(k),x_{2}(k),\cdots ,x_{N}(k)\right] ^{T}\in
\mathbb{R}^{N}.$ By direct algebraic manipulation, the whole system can be
written as
\begin{equation}
x(k+1)=((1+{\theta }_{0})I-{\alpha }\mathcal{L}){x}(k)+\sum\limits_{m=1}^{M}{%
\theta {_{m}}{x}(k-m)},  \label{eq3}
\end{equation}%
with the initial values $x(-M)=\cdots =x(-1)=x(0).$

Denote ${X(k)}=\left[ x(k)^{T},x(k-1)^{T},\cdots ,x(k-M)^{T}\right] ^{T}\in
\mathbb{R}^{N(M+1)}$. Then (\ref{eq3}) becomes
\begin{equation}
X(k+1)=\Phi_M X(k),  \label{X_equ}
\end{equation}%
where
\begin{equation}
\Phi_M =\left[ {%
\begin{array}{ccccc}
{(1+\theta _{0})I-\alpha {\mathcal{L}}} & {\theta _{1}I} & \cdots & {\theta
_{M-1}I} & {\theta _{M}I} \\
I & 0 & \cdots & 0 & 0 \\
0 & I & \cdots & 0 & 0 \\
\vdots & \vdots & \ddots & \vdots & \vdots \\
0 & 0 & \cdots & I & 0%
\end{array}%
}\right] \in \mathbb{R}^{N(M+1)\times N(M+1)},  \label{PhiM}
\end{equation}%
and the intial value in (\ref{X_equ}) is ${X(0)=}\left[
x(0)^{T},x(0)^{T},\cdots ,x(0)^{T}\right]^{T}.$

The compact form (\ref{X_equ}) is essentially the same as that of \cite{[AOC09]},
and the models in \cite{[LM11], [OCR10]} can all be viewed as special cases of (%
\ref{X_equ}).

\begin{lemma}
\label{lemma1}
If $\sum\limits_{m=0}^{M}{\theta _{m}}=0$ and $\mathcal{G}$ is connected,
then $1$ is a simple eigenvalue of $\Phi_M$ with the corresponding left eigenvector
$\varphi _{1}$ and right eigenvector $\vec{1}_{N(M+1)}$,
where
\begin{equation}  \label{eigenl}
\varphi _{1}=\left[ \frac{1}{N}\vec{1}_{N}^{T},-\frac{{\vartheta_{0}}}{N}%
\vec{1}_{N}^{T},-\frac{{\vartheta_{0}+\vartheta_{1}}}{N}\vec{1}%
_{N}^{T},\cdots, -\frac{\sum\limits_{j=0}^{M-1}{\vartheta_{m}}}{N}\vec{1}%
_{N}^{T}\right],
\end{equation}
$\vec{1}_{N(M+1)}$ and $\vec{1}_N$ are all $1$ vectors of dimensions $N(M+1)$ and $N$, respectively, and $\vartheta_{m}=\frac{\theta_{m}}{1-{\sum\limits_{l=0}^{M-1}}{%
\sum\limits_{j=0}^{l}}{\theta_j}}$, for $m=0, 1, \cdots, M-1$.
\end{lemma}

\begin{corollary}
For the system (\ref{X_equ}) with the initial vector ${X(0)=}\left[
x(0)^{T},x(0)^{T},\cdots ,x(0)^{T}\right] ^{T} \in \mathbb{R}^{N(M+1)}$, $%
\lim_{k\rightarrow \infty}X(k)={\bar x}{{\vec{1}}_{N(M+1)}}$ holds for
arbitrary $x(0)\in \mathbb{R}^{N}$ if and only if all the eigenvalues of $\Phi_M
$ are within the unit circle except its simple eigenvalue $1$.
\end{corollary}

The proofs of Lemma 3 and Corollary 1 are given in Appendix.


Similar to the definition in \cite{[XB04],[KF09],[OCR10],[AOC09]}, the
convergence rate of the consensus of the system (\ref{X_equ}) is
defined as 
\begin{equation}
\gamma_M :={\rho }_{s}(\Phi _{M}),  \label{convrate}
\end{equation}%
where ${\rho }_{s}(\Phi _{M})$ is the second largest eigenvalue modulus
of the matrix $\Phi _{M}$.

\begin{remark}
Let $e(k)=x(k)-\bar{x}$ be the consensus error. Obviously $e(k)$ is bounded by $\gamma _{M}$, that is, $\left\Vert e(k)\right\Vert
=O(\gamma _{M}^{k})$ for $k$ large enough. In terms of control theory, $%
\gamma _{M}$ determines the settling time of a system. The smaller the convergence rate $\gamma_{M}$ is, the faster the error $e(k)$ converges.
\end{remark}

Define $\Theta _{M}:=[{\theta _{0}},{\theta _{1}},\cdots ,{\theta _{M}}].$
The goal of acceleration algorithm design is to find the parameters $\alpha $ and $
\Theta _{M}$ such that the convergence rate $\gamma _{M}$ is as small as
possible. Denote $\gamma _{M}^{\ast },\alpha ^{\ast }$ and $\Theta _{M}^{\ast }$
the optimal convergence rate and the corresponding optimal parameters, respectively.
The algorithm design problem can be cast into the following optimization problem
\begin{equation*}
\gamma _{M}^{\ast }=\min_{\{\alpha ,\Theta _{M}\}}\gamma _{M}=\min_{\{\alpha
,\Theta _{M}\}}{\rho }_{s}(\Phi _{M}).
\end{equation*}

Most of current results are presented based on the property of $\Phi _{M}$.
However, the dimension of $\Phi _{M}$ is proportional to the number of
network nodes, which may be very large and difficult to analyze. Moreover,
it is difficult to derive the analytic formulas of the parameters from $\Phi
_{M}.$

\section{Optimization of the Convergence Rate and Some Analytical Solutions}

In this section, we analyze the convergence rate of the consensus under the %
control protocol with $M$-tap memory and formulate the optimization problem.
We will derive, for $M\leq 2$, the analytical formulas of the optimal convergence
rate $\gamma _{M}^{\ast }$ and the corresponding optimal parameters $\alpha ^{\ast }$ and $
\Theta _{M}^{\ast }$.

For a connected graph $\mathcal{G}$, recall that $0={\lambda _{1}<\lambda
_{2}\leq \cdots \leq \lambda _{N}}$ are the eigenvalues of the Laplacian
matrix ${\mathcal{L}}$. Define $h_{i}(z),i=1,\cdots, N$, as follows
\begin{eqnarray}
h_{1}(z) &:=&{z^{M}}-{{\theta _{0}z}}^{M-1}-({{\theta _{0}+\theta _{1})}}%
z^{M-2}-{\cdots -}\left( \sum\limits_{m=0}^{M-1}{{\theta _{m}}}\right),
\label{polyh1} \\
h_{i}(z) &:=&{z^{M+1}}-(1+\theta _{0}-{\alpha \lambda _{i}}){z^{M}-}%
\sum\limits_{m=1}^{M}{{\theta _{m}}z}^{M-m},i=2, \cdots, N,  \label{polyhiM}
\end{eqnarray}%
where $\alpha $ and ${{\theta _{m},m=0,\cdots ,M}}$, are the parameters in (\ref%
{X_equ}), and 
\begin{equation*}
h_{1}(z)(z-1)={z^{M+1}}-(1+\theta _{0}){z^{M}}-\sum\limits_{m=1}^{M}{{\theta
_{m}}z}^{M-m}.
\end{equation*}
The following is the key technical lemma with its proof given in
Appendix.

\begin{lemma}
Let $\Phi _{M}$ be defined by (\ref{PhiM}) with a connected graph
$\mathcal{G}$ and $\sum\limits_{m=0}^{M}{\theta _{m}}=0$.
The characteristic polynomial of $\Phi _{M}$ is given by
\begin{equation}
\det (z{I}-\Phi _{M})=(z-1)\prod\limits_{i=1}^{N}h_{i}(z),  \label{Phidet}
\end{equation}%
with $h_{i}(z),i=1,\cdots ,N$, defined in (\ref{polyh1})-(\ref{polyhiM}).
\end{lemma}

The following results are immediate by combing Corollary 1 and (\ref{convrate})
with Lemma 4.

\begin{theorem}
Consider the system (\ref{agenti})-(\ref{control}) on a connected graph $%
\mathcal{G}$ with $\sum\limits_{m=0}^{M}{\theta _{m}}=0.$ Let $%
h_{i}(z),i=1,\cdots ,N,$ be defined by (\ref{polyh1})-(\ref{polyhiM}). Denote
$\bar{r}(h_{i}(z))$ the maximum modulus root of $h_{i}(z).$ Then we have%

(i) the average consensus is reached asymptotically if and only if $h_{i}(z)$
is stable, i.e. $\bar{r}(h_{i}(z))<1$ for $i=1,\cdots ,N.$

(ii) The convergence rate $\gamma _{M}$ in (\ref{convrate}) can be computed
by
\begin{equation*}
\gamma _{M}=\max_{\{i=1,\cdots ,N\}}\bar{r}(h_{i}(z)).
\end{equation*}
\end{theorem}

Theorem 1 has established the direct link between the convergence rate and the
roots of $N$ polynomials $h_{i}(z)$, $i=1,\cdots,N$. Thus, the optimization
problem of the accelerated consensus can be rewritten as
\begin{equation}
\gamma _{M}^{\ast }=\min_{\{\alpha ,\Theta _{M}\}}{\gamma _{M}}=\min_{\{\alpha
,\Theta _{M}\}}\max_{\{i=1,\cdots ,N\}}\bar{r}(h_{i}(z)).  \label{rootopt}
\end{equation}

Theorem 1 shows that the convergence of the system (\ref{X_equ}) with order
$N(M+1)$ is determined by the $N$ polynomials with the orders no greater than $M+1$. When
the order of memory is much less than the number of agents, i.e. $M\ll N,$ the analysis and computation can be greatly simplified. Notice that the polynomials $h_{i}(z), i=2,\cdots, N$, are a subset of interval polynomials, with the coefficient of the term $z^{M}$ being within $[{\alpha \lambda
_{2}-}1-\theta _{0},{\alpha \lambda _{N}-}1-\theta _{0}].$ For $M\leq 2,$ it
follows from Lemma 1 that the stability of the $N-1$ polynomials $%
h_{i}(z),i=2, \cdots, N$, is determined by the corner polynomials, which in
this case are $h_{2}(z)$ and $h_{N}(z).$ Hence we have the following result.

\begin{theorem}
Consider the system (\ref{agenti})-(\ref{control}) on a connected graph $%
\mathcal{G}$ with $M\leq 2$ and $\sum\limits_{m=0}^{M}{\theta _{m}}=0$. Then
we have

(i) the average consensus is reached asymptotically if and only if $%
h_{1}(z),h_{2}(z)$ and $h_{N}(z)$ are stable,

(ii) the optimal convergence rate $\gamma _{M}^*$ defined in (\ref{rootopt}) can be simplified to
\begin{equation}
\gamma _{M}^{\ast }=\min_{\{\alpha ,\Theta _{M}\}}{\gamma _{M}}=\min_{\{\alpha
,\Theta _{M}\}}\max_{\{i=1,2,N\}}\bar{r}(h_{i}(z)).  \label{rootopt2}
\end{equation}
\end{theorem}

\begin{remark}
For $M\geq 3,$ the stability criterion of Kharitonov Theorem for interval
polynomials does not hold for discrete-time systems. Hence it is not
sufficient to check only three polynomials. This is one of the important
differences between the continuous-time and the discrete-time systems.
\end{remark}

When $M\leq 2,$ we can use Theorem 2 to derive the analytic formulas for
the optimal convergence rate $\gamma _{M}^{\ast }$ and the corresponding optimal
parameters $\alpha ^{\ast }$ and $\Theta _{M}^{\ast }.$ The main idea is as
follows. From Theorem 2, the MAS (\ref{agenti})-(\ref{control}) achieves consensus with convergence rate $\gamma _{M}$ if and only if the maximum modulus root of
the polynomials $h_{i}(z), i=1,2,N$, are within the circle $D(0,\gamma _{M})$,
where $h_{i}(z),i=1,2,N$, can be derived by (\ref{polyh1})-(\ref{polyhiM}).
Using the bilinear transformation $z=\gamma _{M}\frac{s+1}{s-1}$
with $0<\gamma _{M}\leq 1$, the interior of the circle $D(0,\gamma _{M})$ in
the $z$-plane can be mapped to the left half of the $s$-plane. Then
the Routh criterion \cite{[Norman15]} 
can be utilized to formulate the optimization problem.
Finally we can obtain the explicit formulas by direct algebraic manipulations.
The following subsections show details of $M=1$ and $M=2.$

\subsection{Explicit formulas of optimal convergence rate and parameters for MASs with memory $M=1$}

\bigskip In this subsection, we derive the analytic formulas for the optimal
convergence rate and the corresponding optimal parameters for $M=1.$ The control algorithm is given by
\begin{equation}
{u_{i}}(k)={\alpha }\sum\limits_{j\in {N_{i}}}{{a_{ij}}({x_{j}}(k)-{x_{i}}%
(k))}+{{\theta _{0}}{x_{i}}(k)}+{{\theta _{1}}{x_{i}}(k-1)}.
\label{M1order-control}
\end{equation}%
The closed loop system can be written as
\begin{equation}
x(k+1)=({(1+\theta _{0})I-\alpha \mathcal{L)}}{x}(k)+{{\theta _{1}}{x}(k-1)}.
\label{M1order-system}
\end{equation}%
We have the following result.

\begin{theorem}
\bigskip Consider the system (\ref{M1order-system}) on a connected
graph $\mathcal{G}$ with $\theta _{0}+\theta _{1}=0$.
Let $\lambda_2$ and $\lambda_N$ be the smallest and the largest
nonzero eigenvalues of $\mathcal{L}$, respectively.
Then the optimal convergence rate $\gamma_1^{\ast}$ defined by
(\ref{rootopt2}) is given by
\begin{equation}
\gamma_1^{\ast }=\frac{\sqrt{\lambda _{N}/\lambda _{2}}-1}{\sqrt{\lambda
_{N}/\lambda _{2}}+1},  \label{r-value}
\end{equation}
and the corresponding optimal parameters are
\begin{eqnarray}
\alpha ^{\ast } &=& \frac{4}{\left( \sqrt{\lambda _{N}}+\sqrt{\lambda _{2}}%
\right)^{2}},  \label{alpha-value} \\
\theta _{0}^{\ast } &=& \left( \frac{\sqrt{\lambda _{N}/\lambda _{2}}-1}{%
\sqrt{\lambda _{N}/\lambda _{2}}+1}\right) ^{2}  \label{theta0-value}, \\
\theta _{1}^{\ast } &=& -\theta _{0}^{\ast }.  \label{theta1-value}
\end{eqnarray}

\end{theorem}


\begin{remark}
For the memoryless scheme, i.e., the algorithm (\ref{M1order-control}) with $%
\theta _{0}=\theta _{1}=0$, the optimal convergence rate can be obtained as $%
\gamma _{BC}^{\ast }=\frac{\lambda _{N}-\lambda _{2}}{\lambda _{N}+\lambda
_{2}}$ by choosing $\alpha _{0}^{\ast }=\frac{2}{\lambda _{N}+\lambda _{2}}$
in \cite{[XB04]}. Theorem 1 shows that increasing memory by one improves the
convergence rate from $\gamma _{BC}^{\ast }=\frac{\lambda _{N}-\lambda _{2}}{%
\lambda _{N}+\lambda _{2}}$ to $\gamma _{1}^{\ast }=\frac{\sqrt{\lambda
_{N}/\lambda _{2}}-1}{\sqrt{\lambda _{N}/\lambda _{2}}+1}$. For the FIR
memory-enhanced schemes proposed in \cite{[PDK20]}, the algorithm is
\begin{equation}
u_{i}(k)=\sum\limits_{m=0}^{M}{\beta _{m}}\sum\limits_{j\in {N_{i}}}{{a_{ij}}%
({x_{j}}(k-m)-{x_{i}}(k-m))},  \label{FIR-L}
\end{equation}%
and the optimal convergence rate with $M=1$ can be derived as $\gamma
_{FIRMem}^{\ast }=\frac{\lambda _{N}-\lambda _{2}}{3\lambda _{2}+\lambda _{N}%
}$. It can be verified that
\begin{equation*}
\frac{\gamma _{FIRMem}^{\ast }}{\gamma
_{1}^{\ast }}=\frac{(\lambda _{N}-\lambda _{2})(\sqrt{\lambda _{N}}+\sqrt{%
\lambda _{2}})}{(3\lambda _{2}+\lambda _{N})(\sqrt{\lambda _{N}}-\sqrt{%
\lambda _{2}})}=1+\frac{2\sqrt{\lambda _{N}\lambda _{2}}-2\lambda _{2}}{%
3\lambda _{2}+\lambda _{N}}>1.
\end{equation*}
Hence the convergence rate of FIR algorithm (\ref{FIR-L}) is slower (greater) than that of our algorithm (\ref{M1order-control}) with the optimal parameters (\ref{alpha-value})-(\ref{theta1-value}). This means that there is no need to use neighbours' one-tap memory to accelerate consensus, and one can achieve
faster convergence rate only using the one-tap memory of each agent itself.
\end{remark}

For the one-tape memory scheme proposed in \cite{[LM11]}, the optimal
convergence rate can be derived as $\gamma_{Mem}^*=\frac{\rho _{s}(W)}
{1+\sqrt{1-(\rho _{s}(W))^{2}}}$ by choosing $\alpha^* =\frac{1-\sqrt{1-(\rho_{s}(W))^{2}}}{1+\sqrt{1-(\rho _{s}(W))^{2}}}$
in (3), where $\rho _{s}(W)$ is the second largest eigenvalue modulus of the weight matrix $W$,
and $W$ is a row stochastic matrix. Assume the adjacency matrix $\mathcal{A}$ satisfying
$a_{ij}=w_{ij}$ for $i \neq j$. It has $W=I-\mathcal{D}+\mathcal{A}=I-\mathcal{L}$. Thus, we
have $\rho _{s}(W)= \max \{\left\vert 1-\lambda_2 \right\vert, \left\vert 1-\lambda_N
\right\vert \}$, where $\lambda_2$ and $\lambda_N$ are respectively the second smallest and the largest
nonzero eigenvalues of $\mathcal{L}$. 
Then, we have the following result.

\begin{corollary}
\label{crol2}
Consider the MAS in the form of (3) on a connected graph $\mathcal{G}$, with the weight matrix $W$ satisfying $w_{ij}=a_{ij}$ for $i \neq j$. Then the optimal convergence rate derived in \cite{[LM11]} for the system (3) cannot be smaller than $\gamma_1^*$, that is,
\begin{equation}
\gamma_{Mem}^*=\frac{\rho _{s}(W)}{1+\sqrt{1-(\rho _{s}(W))^{2}}} \geq
\gamma_1^*=\frac{\sqrt{\lambda _{N}/\lambda _{2}}-1}{\sqrt{\lambda
_{N}/\lambda _{2}}+1},
\end{equation}
and the equality holds only when $\lambda_2+\lambda_N=2$.
\end{corollary}

\begin{proof}
Since $w_{ij}=a_{ij}$ for $i \neq j$, it has $W=I-\mathcal{L}$ and
$\rho _{s}(W)= \max \{\left\vert 1-\lambda_2 \right\vert, \left\vert 1-\lambda_N
\right\vert \}$. Thus $\rho_{s}(W)$ takes the minimum value when
$\left\vert 1-\lambda_2 \right\vert=\left\vert 1-\lambda_N \right\vert$. Then
we have $\rho_{s}(W) \geq \frac{\lambda_N-\lambda_2}{2}$ with equality for
$\lambda_2+\lambda_N=2$.

As $\rho_{s}(W) \geq \frac{\lambda_N-\lambda_2}{2}$, it is easy to verify that
\begin{equation}
\gamma_{Mem}^*=\frac{\rho _{s}(W)}{1+\sqrt{1-(\rho _{s}(W))^{2}}} \geq
\frac{\lambda _{N}-\lambda _{2}}{2+\sqrt{4-(\lambda_{N}-\lambda _{2})^2}}
\end{equation}
and the equality holds only when $\rho_{s}(W) = \frac{\lambda_N-\lambda_2}{2}$. Therefore,
\begin{eqnarray}
\gamma_{Mem}^*-\gamma_1^* &\geq&
\frac{\lambda _{N}-\lambda _{2}}{2+\sqrt{4-(\lambda_{N}-\lambda _{2})^2}}
-\frac{\sqrt{\lambda _{N}/\lambda _{2}}-1}{\sqrt{\lambda_{N}/\lambda _{2}}+1} \nonumber\\
&=& \frac{(\sqrt{\lambda_N}-\sqrt{\lambda_2})((\sqrt{\lambda_N}+\sqrt{\lambda_2})^2-2
-\sqrt{4-(\lambda_N-\lambda_2)^2})}{(2+\sqrt{4-(\lambda_{N}-\lambda _{2})^2})(\sqrt{\lambda_{N}}+\sqrt{\lambda _{2}})} \nonumber\\
&=& \frac{(\sqrt{\lambda_N}-\sqrt{\lambda_2})(2\sqrt{\lambda_N\lambda_2}-\sqrt{4\lambda_N\lambda_2})}
{(2+\sqrt{4-(\lambda_{N}-\lambda _{2})^2})(\sqrt{\lambda_{N}}+\sqrt{\lambda _{2}})} \nonumber\\
&=& 0.
\end{eqnarray}
This completes the proof.
\end{proof}

\begin{remark}
For the one-tap memory scheme proposed in \cite{[OCR10]}, the optimal convergence rate
is derived as $\gamma_{GMem}^*=\frac{\sqrt{2-(\rho_s(W))^2-2\sqrt{1-(\rho_s(W))^2}}}{\rho_s(W)}$.
It can be verified that
\begin{equation*}
\frac{\gamma_{Mem}^*}{\gamma_{GMem}^*}=\frac{(\rho_s(W))^2}{(1+\sqrt{1-(\rho_s(W))^2})(1-\sqrt{1-(\rho_s(W))^2})}
=\frac{(\rho_s(W))^2}{1-(1-(\rho_s(W))^2)}=1.
\end{equation*}
Hence, the convergence rates in \cite{[LM11]} and \cite{[OCR10]} are essentially the same. It then follows from Corollary \ref{crol2} that $\gamma_1^* \leq \gamma_{Mem}^* = \gamma_{GMem}^*$.
\end{remark}

\subsection{Explicit formulas of optimal convergence rate and parameters for MASs with memory $M=2$}

In this subsection, we derive the analytic formulas for the fastest convergence
rate and the corresponding optimal parameters for $M=2$. The controller with two
memories $M=2$ is as follows
\begin{equation}
{u_{i}}(k)={\alpha }\sum\limits_{j\in {N_{i}}}{{a_{ij}}({x_{j}}(k)-{x_{i}}%
(k))}+{{\theta _{0}}{x_{i}}(k)}+{{\theta _{1}}{x_{i}}(k-1)+{\theta _{2}}{%
x_{i}}(k-2)}.  \label{M2-control}
\end{equation}%
The closed loop system can be written as
\begin{equation}
x(k+1)=({(1+\theta _{0})I-\alpha \mathcal{L)}}{x}(k)+{{\theta _{1}}{x}(k-1){%
+\theta _{2}x}(k-2)}  \label{M2-system}.
\end{equation}%

\begin{theorem}
Consider the system (\ref{M2-system}) on a connected
graph $\mathcal{G}$ with $\theta _{0}+\theta _{1}+\theta _{2}=0$.
Let $\lambda_2$ and $\lambda_N$ be the smallest and the largest
nonzero eigenvalues of $\mathcal{L}$, respectively.
Then the optimal convergence rate $\gamma_2^{\ast}$ defined by
(\ref{rootopt2}) is given by
\begin{equation}
\gamma_2^{\ast }=\frac{\sqrt{\lambda _{N}/\lambda _{2}}-1}{\sqrt{\lambda
_{N}/\lambda _{2}}+1},  \label{r-valueM2}
\end{equation}
and the corresponding optimal parameters are
\begin{eqnarray}
\alpha^{\ast} &=& \frac{4}{\left( \sqrt{\lambda _{N}}+\sqrt{\lambda _{2}}%
\right) ^{2}}  \label{alphaM2} \\
\theta _{0}^{\ast } &=& \left( \frac{\sqrt{\lambda _{N}/\lambda _{2}}-1}{%
\sqrt{\lambda _{N}/\lambda _{2}}+1}\right) ^{2}  \label{theta0M2} \\
\theta _{1}^{\ast } &=& -\theta _{0}^{\ast }  \label{theta1M2} \\
\theta _{2}^{\ast } &=& 0.  \label{theta2M2}
\end{eqnarray}%
\end{theorem}

\begin{remark}
As seen from  (\ref{r-valueM2})-(\ref{theta2M2}) and (\ref{r-value})-(\ref{theta1-value}), $\gamma_2^* = \gamma_1^*$ and all the optimal parameters for $M=2$ and $M=1$ are the same, except
$\theta _{2}^{\ast }=0$ which is the parameter exclusively for $M=2$.
Hence, the optimal controller for $M=2$ has actually used only one-tap memory, and increasing memory from one-tap $M=1$ to two-tap $M=2$ has not yielded a faster convergence rate. 
This seems to suggest that further increasing 
memory 
may not improve convergence rate either.
However, this is not true as shown in the following example.

\textbf{\emph{Example 1:}} Consider a multi-agent system on a star graph with $9$
agents. The nonzero eigenvalues of its Laplacian matrix are $\lambda
_{N}=9,\lambda _{2}=1$. Using the controller (\ref{control}) with $M=3$,
and setting $\alpha =0.258738,\theta _{0}=0.293692,\theta _{1}=-0.301255,\theta
_{2}=0$ and $\theta _{3}=0.007563$, we have
\begin{equation*}
\Phi _{3}=\left[ {%
\begin{array}{cccc}
{(1+0.293692)I-0.258738{\mathcal{L}}} & {-0.301255I} & {0} & {0.007563I} \\
I & 0 & 0 & 0 \\
0 & I & 0 & 0 \\
0 & 0 & I & 0%
\end{array}%
}\right] .
\end{equation*}%
Then we can compute that $\gamma _{3}=\bar{\lambda}(\Phi -\frac{1}{N}\varphi
_{1}\vec{1}_{4N}^{T})=0.3946$. Clearly $\gamma _{3}<\gamma _{1}^{\ast }=\gamma_2^* = \frac{\sqrt{\lambda _{N}/\lambda _{2}}-1}{\sqrt{\lambda _{N}/\lambda _{2}}+1} =0.5$. Hence, increasing 
memory from $M=1$ (or $M=2$) to $M=3$ yields an
improved convergence rate for the system on the star graph.

\end{remark}


Example 1 shows that the relation between the optimal convergence rate and the memory tap $M$ is not trivial, and it may be interesting to study the optimal convergence rate and the corresponding optimal parameters for $M \geq 3$ on fixed graphs.
However, we have found that in the practically worst-case scenario of uncertain graphs, the controller (\ref{M1order-control}) with $M=1$ is sufficient for the optimal convergence rate. We present this finding in the next section.

\section{The optimal worst-case convergence rate on a set of graphs}

In this section, we show that the one-tap memory protocol with the
parameters given in (\ref{alpha-value})-(\ref{theta1-value})
provides the optimal worst-case convergence rate for an uncertain
graph set. Our proof depends on the gain margin optimization of robust
stability stated in Lemma 2.

Let $\{\mathcal{G}\}_{[\underline{\lambda },\bar{\lambda}]}$ be the set of
all connected graphs with $[\lambda _{2},\lambda _{N}]\subseteq \lbrack
\underline{\lambda },\bar{\lambda}].$ Consider the system (\ref{agenti})-(%
\ref{control}) on $\{\mathcal{G}\}_{[\underline{\lambda },\bar{\lambda}]}$
with $\sum\limits_{j=0}^{M}{\theta _{j}}=0$. The worst-case convergence rate
is defined as
\begin{equation*}
\gamma_M =\sup_{G\in \{\mathcal{G}\}_{[\underline{\lambda },\bar{\lambda}%
]}}\rho _{s}(\Phi _{M})=\sup_{G\in \{\mathcal{G}\}_{[\underline{\lambda },%
\bar{\lambda}]}}\max_{\{i=1,\cdots ,N\}}\bar{r}(h_{i}(z)).
\end{equation*}%
The goal is to find the parameters $\alpha $ and $\Theta _{M}$ such that the
worst-case convergence rate is as small as possible, which can be described
by the following optimization problem%
\begin{equation}
\min_{\{\alpha ,\Theta _{M}\}}\gamma_M =\min_{\{\alpha ,\Theta _{M}\}}\sup_{G\in \{%
\mathcal{G}\}_{[\underline{\lambda },\bar{\lambda}]}}\max_{\{i=1,\cdots ,N\}}%
\bar{r}(h_{i}(z)).  \label{Opt_worst}
\end{equation}

\begin{theorem}
Consider the system (5)-(6) on $\{\mathcal{G}\}_{[\underline{\lambda },\bar{%
\lambda}]}$ with $\sum\limits_{j=0}^{M}{\theta _{j}}=0$. For any $M\geq 1,$ the
\ solutions of (\ref{Opt_worst}) are as follows%
\begin{eqnarray}
\gamma_M ^{\ast } &=&\frac{\sqrt{\bar{\lambda}/\underline{\lambda }}-1}{\sqrt{%
\bar{\lambda}/\underline{\lambda }}+1}  \label{solution-worstcase-r} \\
\alpha ^{\ast } &=&\frac{4}{\left( \sqrt{\bar{\lambda}}+\sqrt{\underline{%
\lambda }}\right) ^{2}}  \label{solution-worstcase-a} \\
\theta _{0}^{\ast } &=&\left( \frac{\sqrt{\bar{\lambda}/\underline{\lambda }}%
-1}{\sqrt{\bar{\lambda}/\underline{\lambda }}+1}\right) ^{2}
\label{solution-worstcase-t} \\
\theta _{1}^{\ast } &=&-\theta _{0}^{\ast }  \label{solution-worstcase-ta2}
\\
\theta _{m}^{\ast } &=&0, m=2, \cdots, M.  \notag
\end{eqnarray}
\end{theorem}

Before proving the theorem, we need some notations and a lemma that is proved in Appendix.

Define%
\begin{eqnarray*}
h(z;\lambda ) &:=&{z^{M+1}}-(1+\theta _{0}-{\alpha \lambda }){z^{M}-}%
\sum\limits_{m=1}^{M}{{\theta _{m}}z}^{M-m} \\
&=&(z-1)h_{1}(z)+{\alpha \lambda z}^{M},
\end{eqnarray*}%
where $h_{1}(z)$ is as defined in (\ref{polyh1}). The second equality follows
from the fact that $\sum\limits_{m=0}^{M}{{\theta _{m}=0.}}$

Let $P(z;\lambda )=\lambda \frac{z^{-1}}{1-z^{-1}}$ be a plant and $C(z)=\frac{\alpha z^{M}}{%
h_{1}(z)}$ be the negative feedback controller of the plant. The transfer function of the closed-loop system is given by%
\begin{eqnarray}
T(z;\lambda ) &=&\frac{P(z;\lambda )C(z)}{1+P(z;\lambda )C(z)}=\frac{\alpha
\lambda \frac{z^{-1}}{1-z^{-1}}\frac{z^{M}}{h_{1}(z)}}{1+\alpha \lambda
\frac{z^{-1}}{1-z^{-1}}\frac{z^{M}}{h_{1}(z)}}  \notag \\
&=&\frac{\alpha \lambda z^{M}}{(z-1)h_{1}(z)+\alpha \lambda z^{M}}=\frac{%
\alpha \lambda z^{M}}{h(z;\lambda )}.  \label{Tzlambda}
\end{eqnarray}

\begin{lemma}
The optimal gain margin $k_{\sup}$ of $P(rz;\underline{\lambda })=\underline{\lambda }%
\frac{(rz)^{-1}}{1-(rz)^{-1}}$ is given by $k_{\sup }=\left( \frac{1+r}{1-r}%
\right) ^{2}.$
\end{lemma}

We are now ready to prove Theorem 5.
\begin{proof}
The optimization problem (\ref{Opt_worst}) is equivalent to
\begin{equation}
\min_{\{\alpha ,\Theta _{M}\}}\max \left\{ \bar{r}(h_{1}(z)),\max_{\lambda \in
\lbrack \underline{\lambda },\bar{\lambda}]}\bar{r}(h(z;\lambda ))\right\} .
\label{Opt_worst2}
\end{equation}
Note that all roots of the polynomial polynomial $h(z)$ are within the
circle $D(0,r)$ if and only if $h(rz)$ is stable. Then we have $\bar{r}%
(h(z))=\inf \{r:h(rz)$ is stable$\}.$\bigskip\ Hence, (\ref{Opt_worst2}) can
be written as 
\begin{eqnarray*}
&&\min_{\{\alpha ,\Theta _{M}\}}r \\
\text{s.t. } &&\text{(i) }h_{1}(rz)\text{ is stable; } \\
&&\text{(ii) }h(rz;\lambda )\text{ is stable for any }\lambda \in \lbrack
\underline{\lambda },\bar{\lambda}].
\end{eqnarray*}%
It follows from the representation of $P(z)$ and $T(z)$ that the above
optimization problem is equivalent to 
\begin{eqnarray}
&&\min_{\{\alpha ,\Theta _{M}\}}r  \label{opt_worst3} \\
&&\text{s.t. (i) }C(rz)\text{ is stable, }  \notag \\
&&\text{ \ \ (ii) }T(rz;\lambda )\text{ is stable for any }\lambda \in
\lbrack \underline{\lambda },\bar{\lambda}].  \notag
\end{eqnarray}%
$T(rz;\lambda )$ is stable for any $\lambda \in \lbrack \underline{\lambda },%
\bar{\lambda}]$ means that $T(rz;\underline{\lambda })$ is stable and the
gain margin is at least ${\bar{\lambda}/\underline{\lambda }}.$ If we
relax the stability constraint on $C(rz),$ the problem becomes finding $C(rz)
$ that achieves stability for every $P(rz;\underline{\lambda })$ in $%
\mathcal{P}=\{kP(rz;\underline{\lambda }): 1\leq k\leq {\bar{\lambda}/\underline{\lambda }}\}.$
This means that the optimal gain margin of $P(rz;\underline{\lambda })$
should be at least ${\bar{\lambda}/\underline{\lambda }},$ i.e. $%
k_{\sup }\geq {\bar{\lambda}/\underline{\lambda }}.$ It follows from
Lemma 5 that $k_{\sup }=\left( \frac{1+r}{1-r}\right) ^{2}.\ $Hence we get $%
\left( \frac{1+r}{1-r}\right) ^{2}\geq {\bar{\lambda}/\underline{\lambda }},$
which gives $r\geq \frac{\sqrt{\bar{\lambda}/\underline{\lambda
}}-1}{\sqrt{\bar{\lambda}/\underline{\lambda }}+1}.$ On the other hand, we
know from Theorem 3 that the lower bound $r^{\ast }=\frac{\sqrt{\bar{\lambda}%
/\underline{\lambda }}-1}{\sqrt{\bar{\lambda}/\underline{\lambda }}+1}$ can
be achieved by (\ref{r-value}) when $M=1,$ and the corresponding
$C(r^{\ast}z)=\frac{\alpha^*}{1-\theta_0^*(r^*z)^{-1}}=\frac{\alpha^*}{1-r^*z^{-1}}$
is stable. This completes the proof.
\end{proof}

Theorem 4 shows that for the MAS on any fixed graphs, increasing $M$ from $1$
to $2$ is useless to accelerate the convergence of the system. However, there
might be an MAS with memory of $M > 2$ on a particular graph (for example, a
star graph) that has a faster convergence rate than $\gamma^{\ast}=\frac{\sqrt{{\lambda_N}%
/{\lambda_2}}-1}{\sqrt{{\lambda_N}/{\lambda_2}}+1}$. On the other hand, Theorem 5 shows
that for the MAS on the set of graphs with $\bar{\lambda} \leq \lambda_2 \leq \cdots \leq
\lambda_N \leq \underline{\lambda }$, the one-tap memory algorithm is the optimal one 
in terms of worst-case convergence performance.

The optimal worst-case convergence rate $\gamma_M^{\ast}$ in (\ref{solution-worstcase-r})
actually dose not depend on $M$. The key to obtain such result is to build the relation
between the convergence rate and the optimal gain margin. Recall that in the definition
of the optimal gain margin, the controller achieving $k_{\sup}$ is not necessarily stable.
With Theorem 5, we know that the controller $C(r^{\ast}z)=\frac{\alpha^*}{1-r^*z^{-1}}$
stabilizes all the plant in $\mathcal{P}=\{kP(r^{\ast}z;\underline{\lambda }): 1\leq k\leq
{\bar{\lambda}/\underline{\lambda }}\}$ and is stable itself.

\begin{remark}
For the MAS with a large number of agents, the eigenvalues $\lambda_i$, $i=2,\cdots, N$,
usually distributed densely in $[\lambda_2, \lambda_N]$. If this is the case, one may
expect from Theorem 5 that the one-tap memory controller (\ref{M1order-control})
with parameters given by Theorem 3 presents the optimal convergence rate. Roughly
speaking, the one-tap memory controller with convergence rate $\gamma_1^*$ is a
very good choice for any MAS systems.
\end{remark}

\section{Numerical examples}

This section presents simulation and numerical experiments to show the
effectiveness of the theoretical results.



\begin{figure}
\setcounter{subfigure}{0}
\centering
\subfigure[The small-world graph $\mathcal{G}_5$]{
\includegraphics[scale=0.35]{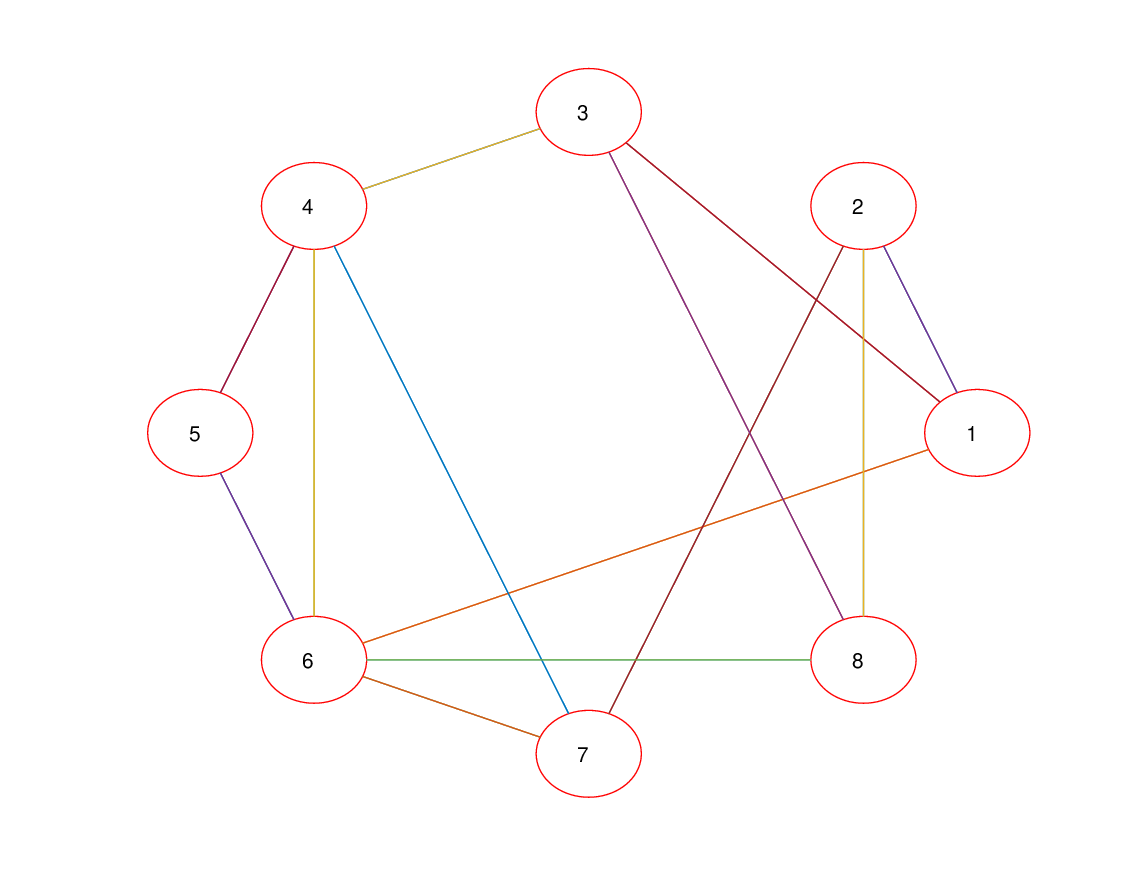}
\label{fig4a}}
\hspace{0.5 cm}
\subfigure[The BA scale-free graph $\mathcal{G}_6$]{
\includegraphics[scale=0.35]{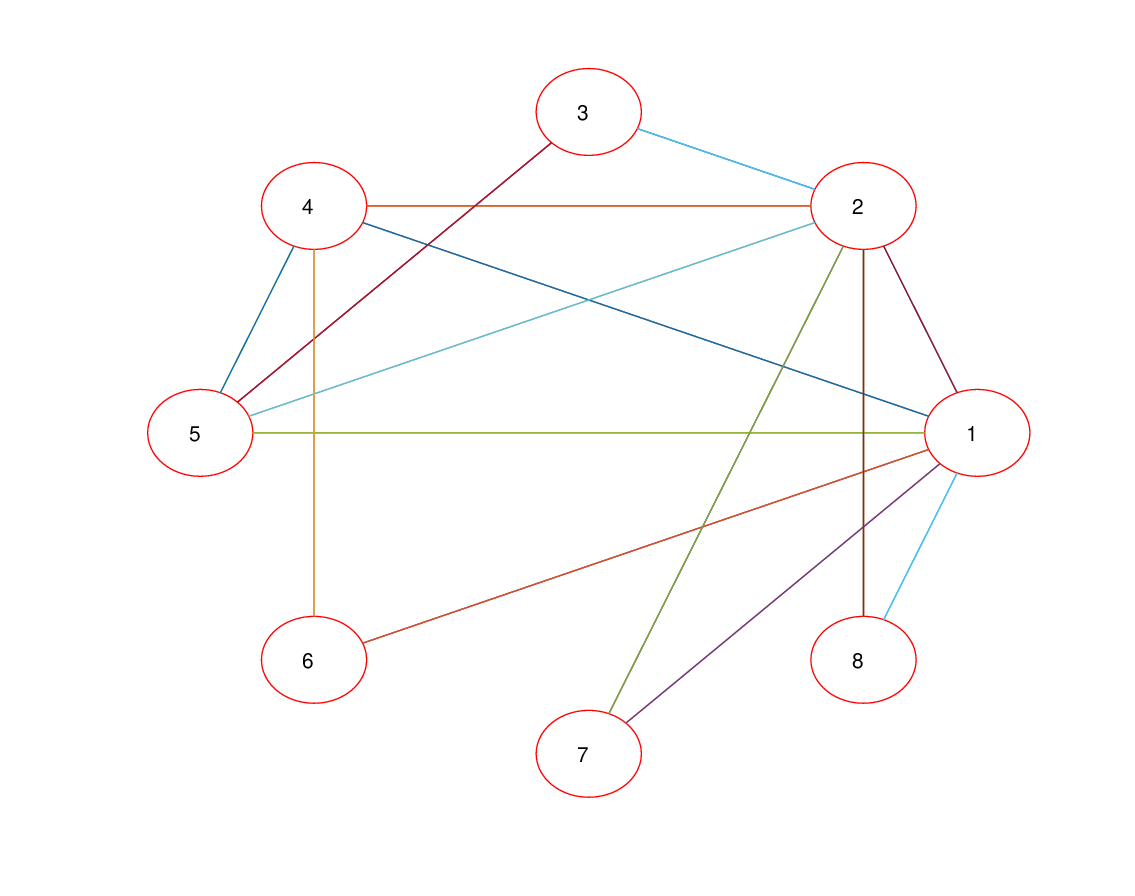}
\label{fig4b}}
  \caption{The small-world graph and the BA scale-free graph.}
  \label{fig4}
\end{figure}

Consider an MAS with 8 agents on six unweighted graphs: (a) the cycle graph $%
\mathcal{G}_{1}$, (b) the path graph $\mathcal{G}_{2}$, (c) the star graph $%
\mathcal{G}_{3}$, (d) the complete bipartite graph $\mathcal{G}_{4}$ with $%
3+5$ vertices, (e) the graph $\mathcal{G}_{5}$ generated by a
small-world network model shown in Fig. \ref{fig4a}, and
(f) the graph $\mathcal{G}_{6}$ generated by a BA scale-free network model
shown in Fig. \ref{fig4b}. In order to compare the consensus performance, we
consider the following algorithms with or without memory: the best
constant gain scheme (BC-L) proposed in \cite{[XB04]}, the graph filtering scheme
with 3-periodic control sequence (GF-L) proposed in \cite{[YCZ20]}, the one-tap memory
scheme (Mem-W) proposed in \cite{[LM11]}, the general one-tap memory scheme
(GMem-W) proposed in \cite{[OCR10]}, the FIR memory-enhanced scheme
(FIRMem-L) proposed in \cite{[PDK20]}, and the optimal one-tap memory scheme
(OptMem-proposed) proposed in this paper. Choosing the entries of the weight
matrix $W$ according to
\begin{equation*}
{w_{ij}}=\left\{ {%
\begin{array}{cc}
{\frac{1}{\max \{d_{i},d_{j}\}}} & {(i,j)\in \mathcal{E}}, \\
{1-\sum\limits_{j\in {\mathcal{N}_{i}}}{{w_{ij}}}} & {i=j}, \\
0 & {else},%
\end{array}%
}\right.
\end{equation*}%
it leads to the Metropolis-Hastings weight matrix. Table I shows the
comparison among the algorithms discussed above.

\begin{table*}[tbp]
\caption{Comparison of consensus algorithms}
\label{algorithm}
\centering \makegapedcells \setcellgapes{0pt} \newsavebox{\mybox} %
\newcolumntype{X}[1]{>{\begin{lrbox}{\mybox}}c<{\end{lrbox}\makecell[#1]{%
\usebox\mybox}}}
\begin{tabular}{X{cc}|X{cc}X{cc}}
\hline
\makecell[bc]{Algorithm} & \makecell[bc]{Optimal control parameters} & \makecell[bc]{Convergence rate} \\
\hline
\makecell[bc]{BC-L \cite{[XB04]}} & $\varepsilon^*=\frac{2}{\lambda_2+\lambda_N}$ &
$\gamma_{BC}^*=\frac{\lambda_N-\lambda_2}{\lambda_N+\lambda_2} $ \\
\makecell[bc]{GF-L \cite{[YCZ20]}} & $\varepsilon^*(k+3j)=\frac{2}{{(\lambda_N-\lambda_2)}{\cos{\frac{{2k-1}}{{6}}\pi}}+
{(\lambda_N+\lambda_2)}}$, ${{\textstyle{{k = 1,2,3}, {j = 0,1, \cdots }}}}$ &
$\gamma_{GF}^*=\sqrt[3]{\frac{2}{\left(1-\frac{2}{\sqrt{\lambda_N/\lambda_2}+1}\right) ^{3}
+\left(1+\frac{2}{\sqrt{\lambda_N/\lambda_2}-1}\right) ^{3}} }$ \\
\makecell[bc]{Mem-W \cite{[LM11]}} & $\alpha^*=\frac{\sqrt{1-\rho_s(W)^2}-1}{\sqrt{1-\rho_s(W)^2}+1}$ &
$\gamma_{Mem}^*=\frac{\rho_s(W)}{1+\sqrt{1-\rho_s(W)^2}}$ \\
\makecell[bc]{GMem-W \cite{[OCR10]}} & $\beta_1^*=-\epsilon$, $\beta_2^*=0$, $\beta_3^*=1+\epsilon$,
$\alpha^*=\frac{2-\rho_s(W)^2-2\sqrt{1-\rho_s(W)^2}}{\epsilon\rho_s(W)^2}$ &
$\gamma_{GMem}^*=\frac{\sqrt{2-\rho_s(W)^2-2\sqrt{1-\rho_s(W)^2}}}{\rho_s(W)}$ \\
\makecell[bc]{FIRMem-L \cite{[PDK20]}} &
$\beta_0^*=\frac{\lambda_2+3\lambda_N}{\lambda_N(\lambda_N+3\lambda_2)}$,
$\beta_1^*=\frac{(\lambda_2-\lambda_N)^2}{\lambda_N(\lambda_N+3\lambda_2)}$ &
$\gamma_{FIRMem}^*=\frac{\lambda_N-\lambda_2}{3\lambda_2+\lambda_N}$ \\
\makecell[bc]{OptMem-proposed} & $\alpha^*=\frac{4}{(\sqrt{\lambda_N}+\sqrt{\lambda_2})^2}$,
$\theta_0^*=(\frac{\sqrt{\lambda _{N}/\lambda _{2}}-1}{\sqrt{\lambda_{N}/\lambda _{2}}+1})^2$,
$\theta_1^*=-\theta_0^*$ &
$\gamma_{OptMem}^*=\frac{\sqrt{\lambda _{N}/\lambda _{2}}-1}{\sqrt{\lambda_{N}/\lambda _{2}}+1}$ \\
\hline
\end{tabular}
\end{table*}


Let the consensus error be defined as ${\frac{{\left\Vert {x(t)-\bar{x}}%
\right\Vert }_{2}}{{\left\Vert {x(0)-\bar{x}}\right\Vert }_{2}}}$,
and the tolerable error as $\epsilon =10^{-6}$. We now
investigate the consensus performances of the above algorithms on six
graphs.

\textbf{\emph{1) The cycle graph $\mathcal{G}_1$:}} The eigenratio of the
graph Laplacian $\mathcal{L}_{\mathcal{G}_1}$ is $\frac{\lambda_2}{\lambda_N}%
=0.1464$, and the second largest eigenvalue of the Metropolis-Hastings
weight matrix $W_{\mathcal{G}_1}$ is $\rho_s(W_{\mathcal{G}_1})=1$. It can be derived that $\gamma_{Mem}^*
=1 $ and $\gamma_{GMem}^* =1$, which means the algorithms proposed in \cite%
{[LM11]} and \cite{[OCR10]} cannot reach consensus. The
consensus error trajectories of MAS by different algorithms on $\mathcal{G}%
_1 $ are shown in Fig. \ref{fig1a}. It can be seen that the MAS under the
BC-L scheme in \cite{[XB04]}, the graph filtering scheme in \cite{[YCZ20]},
the FIRMem-L scheme in \cite{[PDK20]} and our
proposed method can achieve consensus. It
appears that our proposed method has the fastest convergence rate.

\textbf{\emph{2) The path graph $\mathcal{G}_2$:}} The eigenratio of the
graph Laplacian $\mathcal{L}_{\mathcal{G}_2}$ is $\frac{\lambda_2}{\lambda_N}%
=0.0396$, and the second largest eigenvalue of the Metropolis-Hastings
weight matrix $W_{\mathcal{G}_2}$ is $\rho_s(W_{\mathcal{G}_2})=0.9239$.
The consensus error trajectories of
MAS by different algorithms on $\mathcal{G}_2$ are shown in Fig. \ref{fig1b}%
. It can be seen that the MAS under the six tested algorithms can
achieve consensus, and the consensus performance by the Mem-W scheme in \cite%
{[LM11]}, GMem-W scheme in \cite{[OCR10]} and our proposed method are the
same with the fastest convergence rate.

\textbf{\emph{3) The star graph $\mathcal{G}_3$:}} The eigenratio of the
graph Laplacian $\mathcal{L}_{\mathcal{G}_3}$ is $\frac{\lambda_2}{\lambda_N}%
=0.1250$, and the second largest eigenvalue of the Metropolis-Hastings
weight matrix $W_{\mathcal{G}_3}$ is $\rho_2(W_{\mathcal{G}_3})=0.8571$.
The consensus error trajectories of
MAS by different algorithms on $\mathcal{G}_3$ are shown in Fig. \ref{fig1c}%
. It can be seen that the memory-enhanced algorithms have the faster
convergence rates than the BC-L scheme in \cite{[XB04]}, and our proposed method
has the fastest convergence rate. Moreover, the FIRMem-L scheme in
\cite{[PDK20]} is the worst among the memory-enhanced algorithms.
This shows that for the one-tap memory-enhanced schemes,
neighbours' memory is not needed to accelerate consensus.
The fastest convergence rate can be achieved
by using one-tap memory of each agent itself.

\textbf{\emph{4) The complete bipartite graph $\mathcal{G}_4$:}} The
eigenratio of the graph Laplacian $\mathcal{L}_{\mathcal{G}_4}$ is $\frac{%
\lambda_2}{\lambda_N}=0.3750$, and the second largest eigenvalue of the
Metropolis-Hastings weight matrix $W_{\mathcal{G}_4}$ is
$\rho_2(W_{\mathcal{G}_4})=0.6000$. The consensus
error trajectories of MAS by different algorithms on $\mathcal{G}_4$ are
shown in Fig. \ref{fig1d}. Although the convergence rate of
the FIRMem-L scheme in \cite{[PDK20]} is faster than that of the
Mem-W scheme in \cite{[LM11]} and the GMem-W scheme in \cite{[OCR10]}, it is
still slower than the optimal scheme proposed in this paper.

\textbf{\emph{5) The graph $\mathcal{G}_5$ generated by a small-world
network model:}} The eigenratio of the graph Laplacian $\mathcal{L}_{%
\mathcal{G}_5}$ is $\frac{\lambda_2}{\lambda_N}=0.2201$, and the second
largest eigenvalue of the Metropolis-Hastings weight matrix $W_{\mathcal{G}_5}$ is $%
\rho_2(W_{\mathcal{G}_5})=0.7211$. The consensus error trajectories of MAS by
different algorithms on $\mathcal{G}_5$ are shown in Fig. \ref{fig1e}.
It can be seen that the convergence rate of the graph filtering scheme in \cite{[YCZ20]}
is faster than that of the FIRMem-L scheme in \cite{[PDK20]}, and our proposed
method has the fastest convergence rate.

\textbf{\emph{6) The graph $\mathcal{G}_5$ generated by a BA scale-free
network model:}} The eigenratio of the graph Laplacian $\mathcal{L}_{%
\mathcal{G}_6}$ is $\frac{\lambda_2}{\lambda_N}=0.2121$, and the second
largest eigenvalue of the Metropolis-Hastings weight matrix $W_{%
\mathcal{G}_6}$ is $\rho_2(W_{\mathcal{G}_6})=0.7105$.
The consensus error trajectories of MAS by
different algorithms on $\mathcal{G}_6$ are shown in Fig. \ref{fig1f}.
It can be seen that the optimal one-tap memory scheme proposed in this paper
always has the fastest convergence rate.

\begin{figure*}[tbp]
\setcounter{subfigure}{0} \centering 
\subfigure[The consensus error trajectories on $\mathcal{G}_1$]{
\label{fig1a}
\includegraphics[scale=0.35]{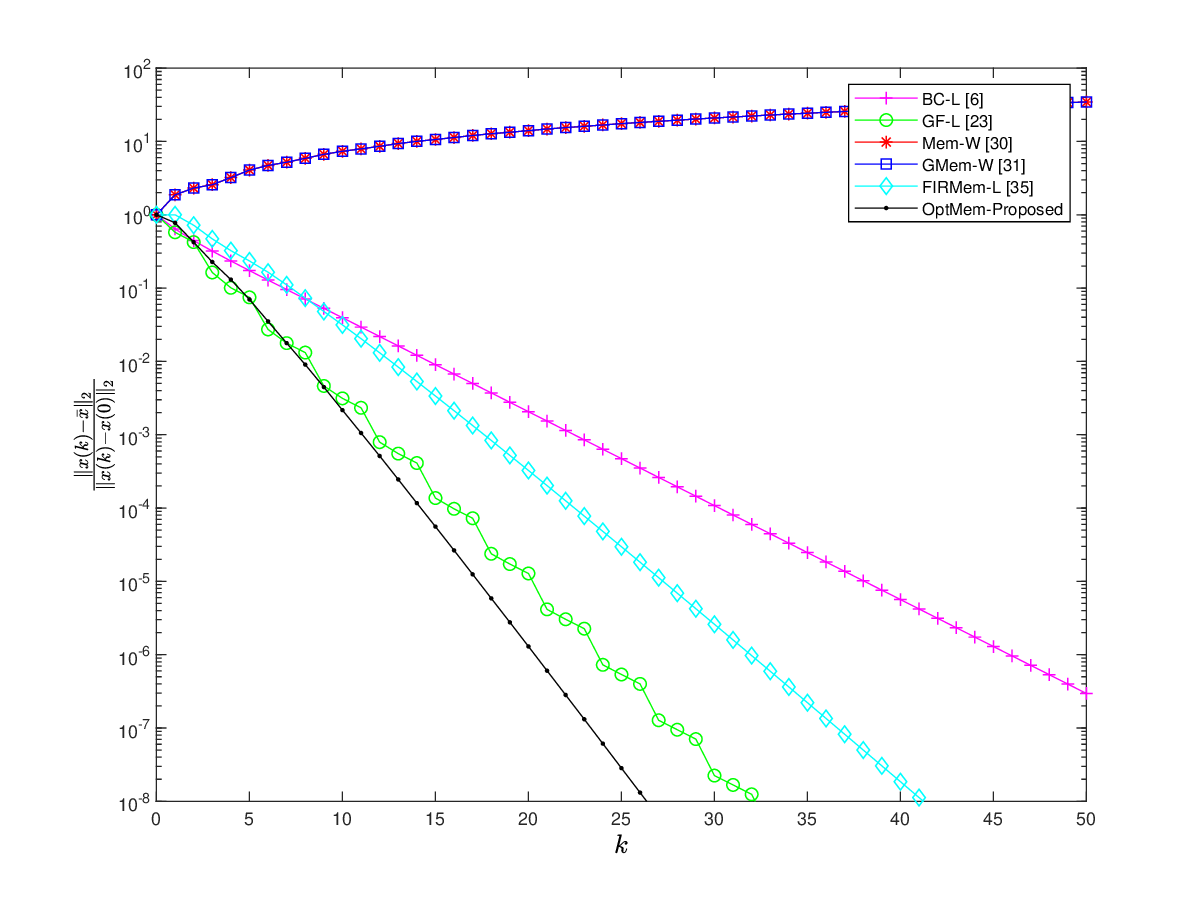}}
\hspace{0.01 cm}
\subfigure[The consensus error trajectories on $\mathcal{G}_2$]{
\label{fig1b}
\includegraphics[scale=0.35]{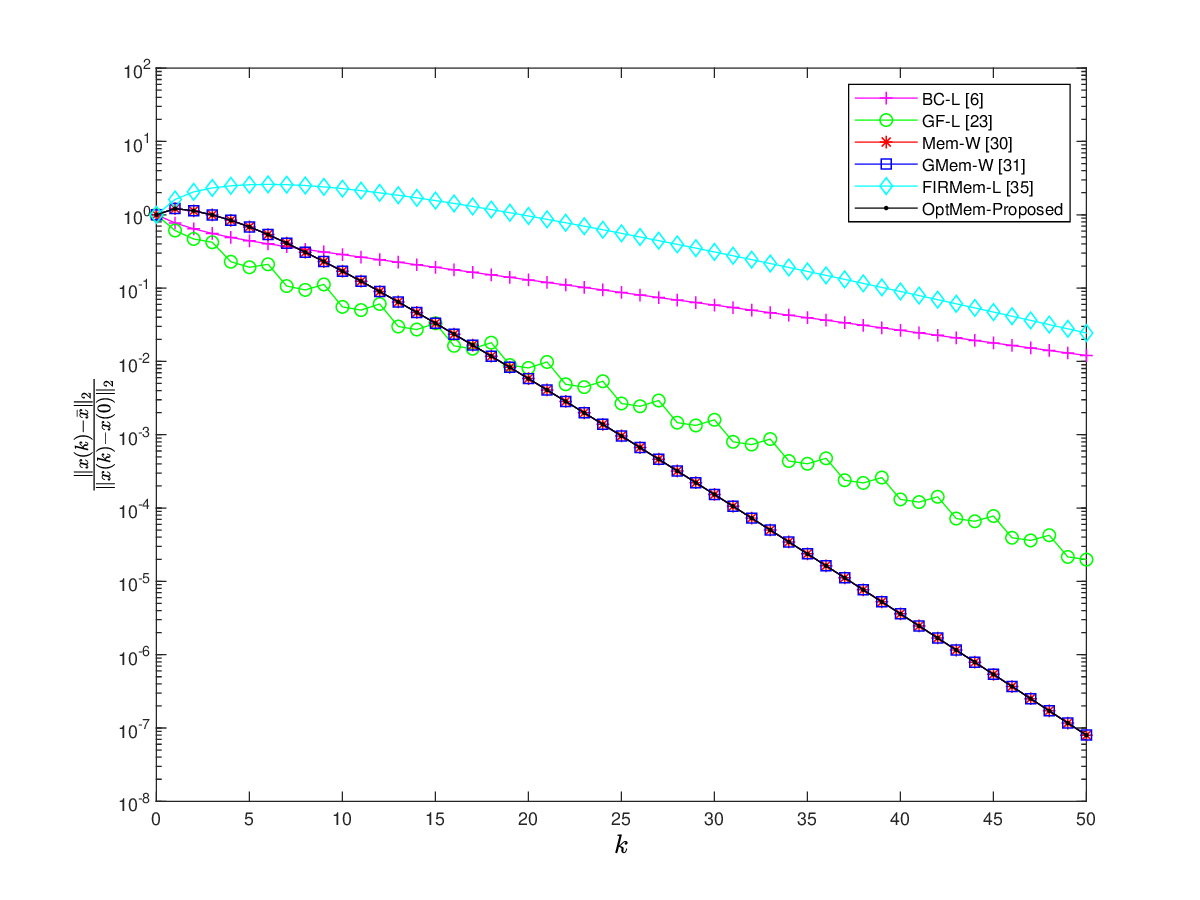}}
\newline
\subfigure[The consensus error trajectories on $\mathcal{G}_3$]{
\label{fig1c}
\includegraphics[scale=0.35]{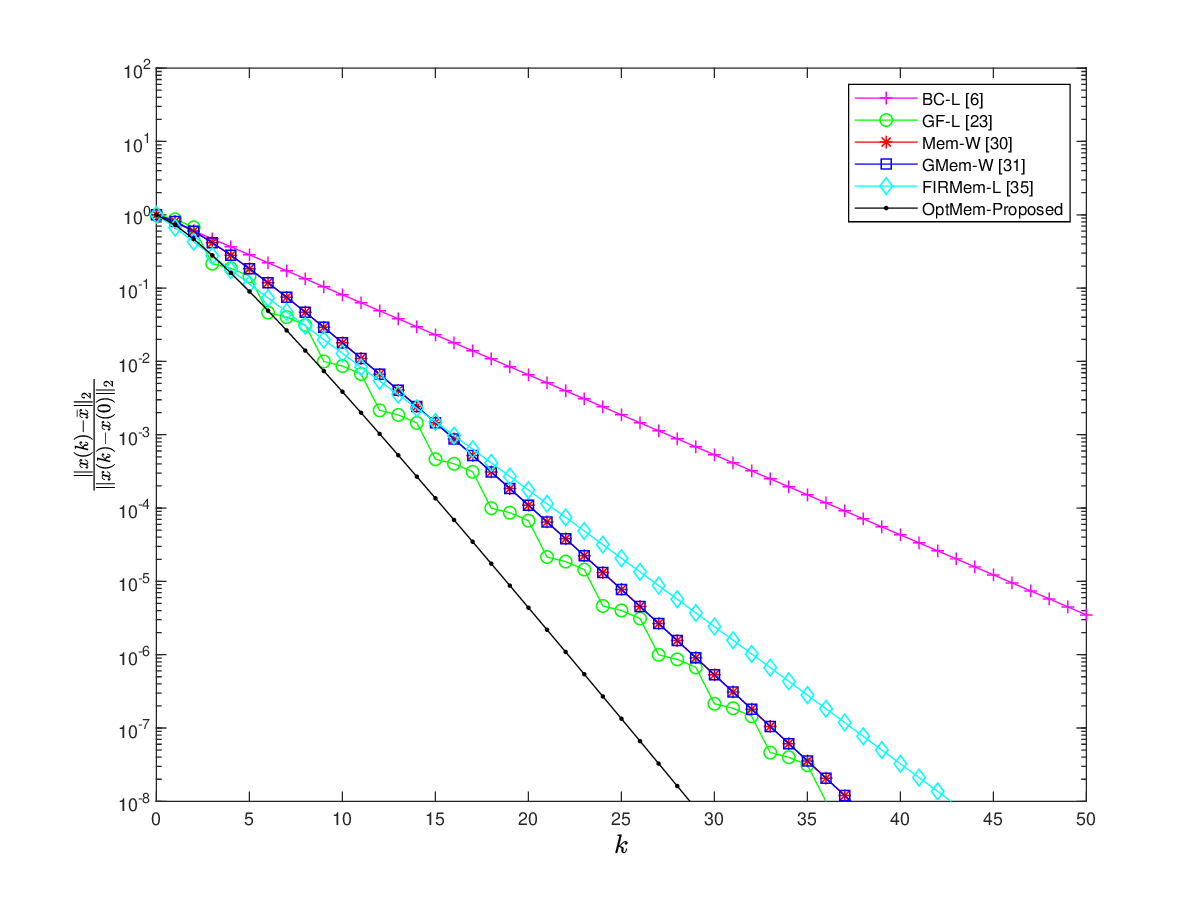}}
\hspace{0.01 cm}
\subfigure[The consensus error trajectories on $\mathcal{G}_4$]{
\label{fig1d}
\includegraphics[scale=0.35]{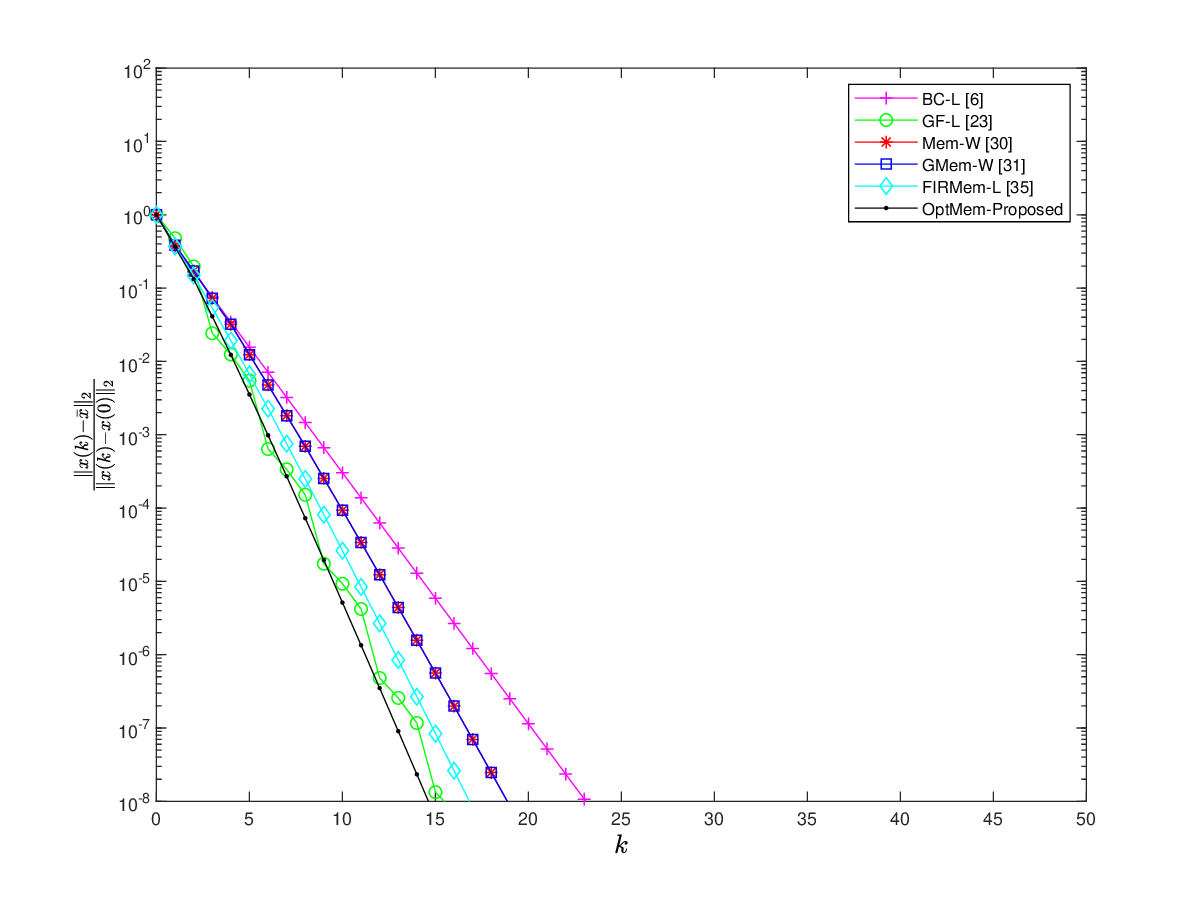}}
\newline
\subfigure[The consensus error trajectories on $\mathcal{G}_5$]{
\label{fig1e}
\includegraphics[scale=0.35]{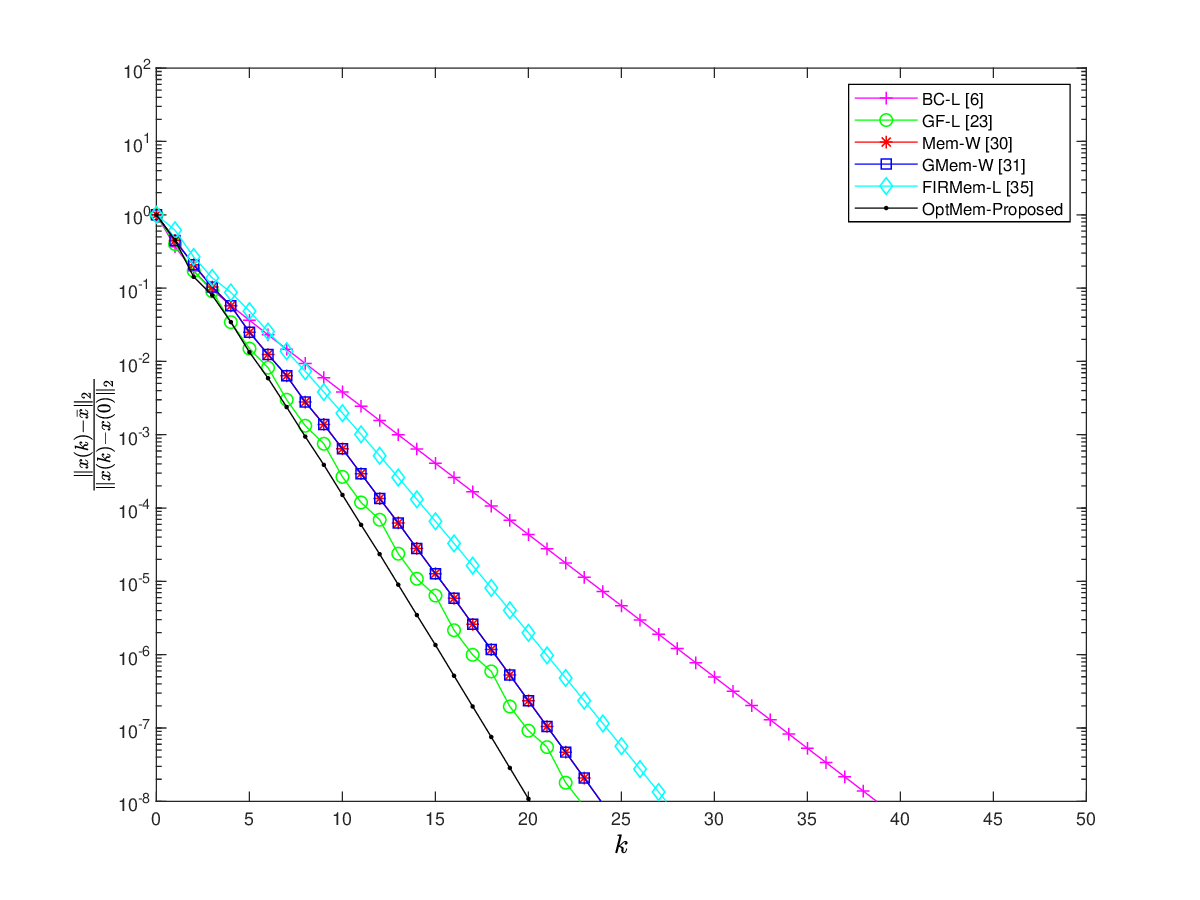}}
\hspace{0.01 cm}
\subfigure[The consensus error trajectories on $\mathcal{G}_6$]{
\label{fig1f}
\includegraphics[scale=0.35]{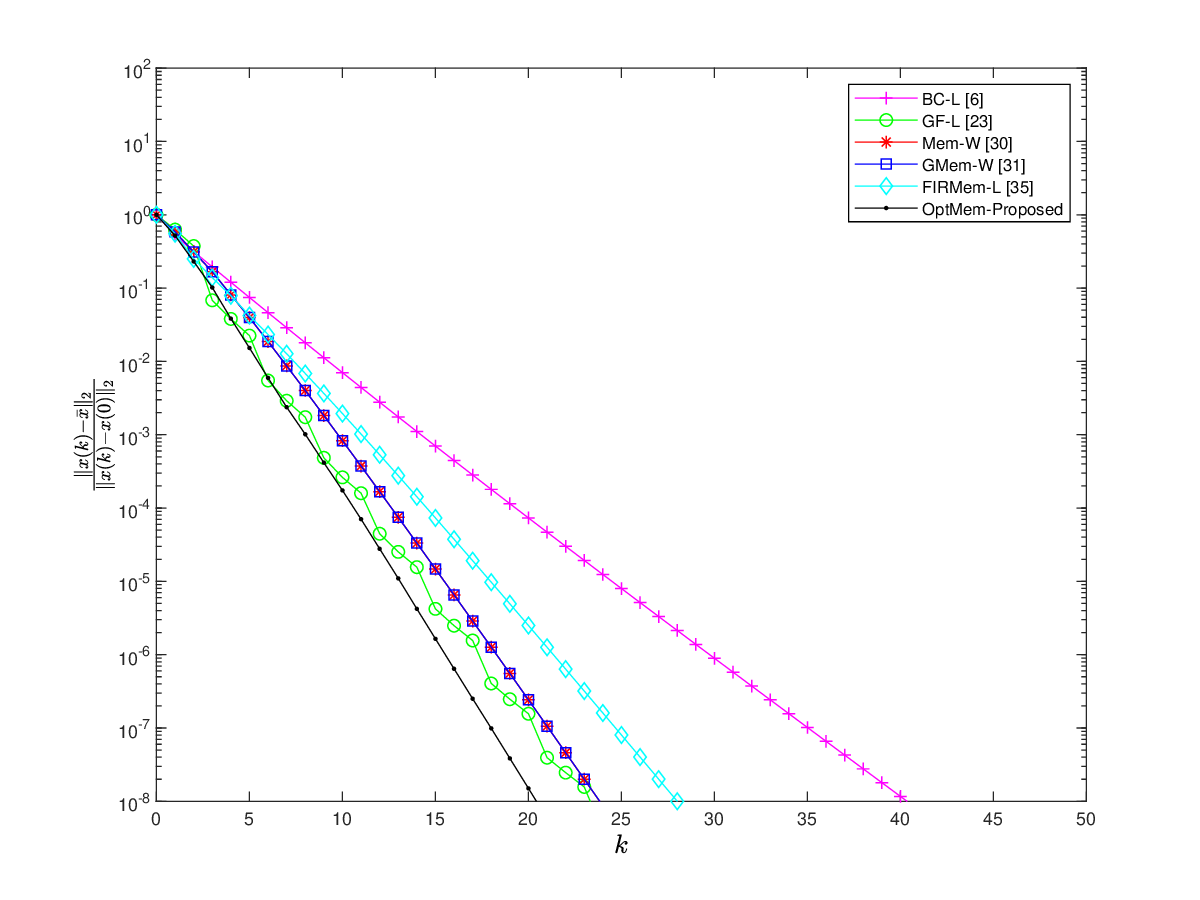}}
\caption{The consensus error trajectories by different methods.}
\label{fig1}
\end{figure*}

Table II gives the exact values of the convergence rates on the six example
graphs by the six tested consensus algorithms.
It can be seen that the convergence rates of the Mem-W scheme in \cite{[LM11]}
and the GMem-W scheme in \cite{[OCR10]} are both determined by the construction
of the weight matrix $W$.
If $W$ is improperly chosen, the MAS will diverge, such as the cycle graph $\mathcal{G}_1$.
If $W$ is chosen appropriately, the convergence rates of the Mem-W scheme in \cite{[LM11]}
and the GMem-W scheme in \cite{[OCR10]} are the same as that of our proposed optimal
scheme, such as the path graph $\mathcal{G}_2$.
Besides, although the algorithms in \cite{[LM11]} and \cite{[OCR10]} are different,
the convergence rate of the two methods are exactly the same, which is faster than the
FIRMem-L scheme in \cite{[PDK20]} for most cases.
Compared to existing algorithms, the optimal algorithm proposed in this paper
always has the fastest convergence rate, which corroborates the analyses in Remarks 3-4.

\begin{table*}[tbp]
\caption{Convergence rate by the discussed algorithms on six sample graphs}
\label{Comp_gamma}
\centering \makegapedcells \setcellgapes{0pt} 
\newcolumntype{X}[1]{>{\begin{lrbox}{\mybox}}c<{\end{lrbox}\makecell[#1]{%
\usebox\mybox}}}
\begin{tabular}{X{cc}|X{cc}X{cc}X{cc}X{cc}X{cc}X{cc}}
\hline
           & \makecell[bc]{Cycle graph} $\mathcal{G}_1$ & \makecell[bc]{Path graph} $\mathcal{G}_2$ &
\makecell[bc]{Star graph} $\mathcal{G}_3$ & \makecell[bc]{Bipartite graph} $\mathcal{G}_4$ & \makecell[bc]{SW graph} $\mathcal{G}_5$ & \makecell[bc]{BA graph} $\mathcal{G}_6$ \\
\hline
$\frac{\lambda_2}{\lambda_N}$ & 0.1464 & 0.0396 & 0.1250 & 0.3750 & 0.2201 &0.2121 \\
$\rho_2(W)$ & 1 & 0.9239 & 0.8571 & 0.6000 & 0.7211 & 0.7105 \\
$\gamma_{BC}^*$ & 0.7445 & 0.9239 & 0.7778 & 0.4545 & 0.6392 & 0.6501 \\
$\gamma_{GF}^*$ & 0.5610 & 0.8183 & 0.5994 & 0.3029 & 0.4549 & 0.4650 \\
$\gamma_{Mem}^*$ & \makecell[bc]{diverse} & 0.6682 & 0.5657 & 0.3333 & 0.4260 & 0.4170 \\
$\gamma_{GMem}^*$ & \makecell[bc]{diverse} & 0.6682 & 0.5657 & 0.3333 & 0.4260 & 0.4170 \\
$\gamma_{FIRMem}^*$ & 0.5930 & 0.8585 & 0.6364 & 0.2941 & 0.4697 & 0.4815 \\
$\gamma_{OptMem}^*$ & 0.4465 & 0.6682 & 0.4776 & 0.2404 & 0.3613 & 0.3694 \\
\hline
\end{tabular}
\end{table*}

\section{Conclusion}

We have introduced a set of effective and previously unused techniques for the analysis and design of the multiple agent average consensus with local memory. Using these techniques, we have analyzed how fast the MAS converges by using local memory and how to design the consensus protocol to achieve the fastest convergence rate. By Kharitonov theorem, we have shown that the convergence rate equals the maximum modulus root of only three polynomials for any $N$ agents when $M \leq 2$. The analytic formulas of the optimal convergence rate and the corresponding optimal parameters have been derived. We have also compared the existing results with the optimal ones presented in this paper. For $M \geq 3$, we have shown that the acceleration with $M=1$ presents the optimal convergence rate in the sense of worst-case performance by using gain margin optimization of robust stability. Numerical experiments have demonstrated the validity, effectiveness, and advantages of these results and methods.

These new results have provided deeper insight into the MAS average consensus with local memory, and effective tools for the design and quantitative performance evaluation of the consensus protocols.

The analysis techniques of this paper can also be applied to directed graphs as long as the eigenvalues of the Laplacian matrix are real. Since the average consensus can be viewed as a special case of the distributed optimization with objective function $\sum\limits_{i=1}^{N}{(x_i-\bar x})^2$, we are currently working on extending the presented results to acceleration algorithms of general distributed optimization problems.

\section{Appendix}

\subsection{Proof of Lemma 3}



%
As $\sum\limits_{m=0}^{M}{\theta _{m}}=0$, it is easy to verify that $\Phi
_{M}\vec{1}_{N(M+1)}=\vec{1}_{N(M+1)}$, where $\vec{1}_{N(M+1)}$ is an all 1 vector of dimension $N(M+1)$, thus $\vec{1}_{N(M+1)}$ is the right eigenvector corresponding to $1$. If $\mathcal{G}$ is connected, then $0$ is the simple eigenvalue of the graph Laplacian matrix $\mathcal{L}$, and $1$ is a simple eigenvalue of $\Phi _{M}$. Let $\varphi _{1}=c\left[ \varphi
_{11},\varphi _{12},\cdots ,\varphi _{1(M+1)}\right] \in \mathbb{R}^{1\times
N(M+1)}$, be the left eigenvector corresponding to $1$, where $\varphi
_{1i}\in \mathbb{R}^{1\times N}$, $i=1,\cdots ,M+1$, it has $\varphi
_{1}\Phi _{M}=\varphi _{1}1$. 
Then, we have
\begin{equation}
\left\{ {%
\begin{array}{c}
{(1+{\theta _{0}}){\varphi _{11}}-{\alpha }{\mathcal{L}}{\varphi _{11}}+{%
\varphi _{12}}={\varphi _{11}}}, \\
{{\theta _{1}}{\varphi _{11}}+{\varphi _{13}}={\varphi _{12}}}, \\
\vdots \\
{{\theta _{M-1}}{\varphi _{11}}+{\varphi _{1(M+1)}}={\varphi _{1M}}}, \\
{{\theta _{M}}{\varphi _{11}}={\varphi _{1(M+1)}}}.%
\end{array}%
}\right.
\end{equation}%
One solution of the above linear equations can be obtained as $\varphi
_{11}=\frac{1}{N}\vec{1}_{N}$, $\varphi _{12}=-{\theta _{0}}\frac{1}{N}\vec{1%
}_{N}$, $\varphi _{13}=-({\theta _{0}}+{\theta _{1}})\frac{1}{N}\vec{1}_{N}$%
, $\cdots $, $\varphi _{1(M+1)}=-{\sum\limits_{m=0}^{M-1}{\theta _{m}}}\frac{%
1}{N}\vec{1}_{N}$, where $\vec{1}_{N}$ is an all 1 vector of dimension $N$.
As $\varphi _{1}\vec{1}_{N(M+1)}=c\sum\limits_{l=1}^{M+1}{%
\varphi _{1l}\vec{1}_{N}}=1$, it has $c=\frac{1}{1-{\sum\limits_{l=0}^{M-1}}{%
\sum\limits_{j=0}^{l}}{\theta _{j}}}$. Thus, the explicit formula of $%
\varphi _{1}$ can be derived as shown in (\ref{eigenl}). This complete the proof.

\subsection{Proof of Corollary 1}
1) Sufficiency: Assume that $\Phi _{M}$ has a simple eigenvalue equal to $1$%
, and the corresponding right eigenvector is $\upsilon _{1}=\vec{1}_{N(M+1)}$%
. Then $\Phi _{M}$ can be written in the Jordan canonical form as
\begin{equation}
\Phi _{M}=V\left[ {%
\begin{array}{cc}
{1} &  \\
& {J}%
\end{array}%
}\right] V^{-1},
\end{equation}%
where $V=\left[ \upsilon _{1},\cdots ,\upsilon _{N(M+1)}\right] $, and the
Jordan block matrix $J\in \mathbb{R}^{(N(M+1)-1)\times (N(M+1)-1)}$
corresponding to the eigenvalues of $\Phi _{M}$ within the unit circle.
Therefore, one can obtain
\begin{equation}
\lim_{k\rightarrow \infty }\Phi _{M}^{k}=V\left[ {%
\begin{array}{cc}
{1} &  \\
& {O_{N(M+1)-1}}%
\end{array}%
}\right] V^{-1}.
\end{equation}%
%
%
%
Then the consensus state of (\ref{X_equ}) are given by
\begin{equation}
\lim_{k\rightarrow \infty }X(k)=\upsilon _{1}\varphi _{1}X(0)=\frac{1}{N}%
\sum\limits_{i=1}^{N}x_{i}(0)\vec{1}_{NM}=\bar{x}\vec{1}_{NM},
\end{equation}%
that is, $\lim_{k\rightarrow \infty }x_{i}(k)=\bar{x}$. %
Then the MAS (\ref{eq3}) reaches average consensus asymptotically.


2) Necessity:
If the MAS achieves average consensus, i.e., $\lim_{k\rightarrow \infty}{%
x_i(k)}=\lim_{k\rightarrow \infty}{x_j(k)}=\bar x$, it has $%
0=\lim_{k\rightarrow \infty}{u_i(k)}={\sum\limits_{m=0}^{M}{\theta_{m}}%
\lim_{k\rightarrow \infty}{x_{i}(k-m)}}=\bar x \sum\limits_{m=0}^{M}{\theta_{m}}$.
Hence $\sum\limits_{m=0}^{M}{\theta _{m}}=0$.
Set $X(0)=\vec{1}_{N(M+1)}$, it has $\Phi_M X(0)=\vec{1}_{N(M+1)}$.
Thus the matrix $%
\Phi_M $ has at least one eigenvalue equal to $1$. If the MAS can reach
average consensus, that is, $\left|X(k)-{\bar x}\vec{1}_{N(M+1)}\right|\to 0$
as $k \to \infty$, then $\Phi_M^k$ must have rank one as $k \to
\infty$, which in turn implies that $J^K$ equals a zero matrix as $k \to
\infty$. It follows that all eigenvalues of $\Phi_M$ are within the unit
circle except that $1$ is a simple eigenvalue. This complete the proof.


\subsection{Proof of Lemma 4}

To prove Lemma 4, we first introduce Schur's Formula \cite{[BV04]}:\\
For a matrix $\mathcal{M}=\left[ {%
\begin{array}{cc}
R & S \\
P & Q%
\end{array}%
}\right] $, with $R\in \mathbb{R}^{n\times n}$, $S\in \mathbb{R}^{n\times m}$%
, $P\in \mathbb{R}^{m\times n}$, $Q\in \mathbb{R}^{m\times m}$ and 
non-singular,
\begin{equation}
\det \left\{ \mathcal{M}\right\} =\det \{Q\}\det \{R-SQ^{-1}P\}.
\end{equation}
Then we prove the lemma.

The eigenvalues of $\Phi_{M}$ are the roots of its characteristic polynomials
$\det (z{I}-\Phi_{M})=0.$ To compute the determinant $\det (z{I}-\Phi_{M})$,
we denote
\begin{eqnarray*}
R &=&{z{I}-{(1+\theta _{0})I}+\alpha {\mathcal{L}}}, \\
S &=&\left[
\begin{array}{cccc}
{-{\theta _{1}}I} & {-{\theta _{2}}I} & \cdots & {-{\theta _{M}}I}%
\end{array}%
\right] , \\
P &=&\left[
\begin{array}{cccc}
{-I} & {0} & \cdots & {0}%
\end{array}%
\right] ^{T}, \\
Q &=&\left[ {%
\begin{array}{ccccc}
{z{I}} & {0} & \cdots & {0} & {0} \\
{-{I}} & {z{I}} & \cdots & {0} & {0} \\
\vdots & \vdots & \ddots & \vdots & \vdots \\
{0} & {0} & \cdots & {-{I}} & {z{I}}%
\end{array}%
}\right] .
\end{eqnarray*}%
It then follows from Schur's Formula that
\begin{equation}
\begin{array}{c}
\det (z{I}-\Phi_{M})={z^{N(M+1)}}\det \{R-\sum\limits_{m=1}^{M}{{\theta _{m}}%
{z^{(-m)}}I}\} \\
=\det \{{z^{M}}((z{I}-{(1+\theta _{0})I}+{\alpha \mathcal{L}}%
)-\sum\limits_{j=m}^{M}{{\theta _{m}}{z^{(-m)}}I})\} \\
=\det \{z^{M+1}{I}-({I}+{\theta _{0}}{I}-{\alpha }\mathcal{L}){z^{M}}%
-\sum\limits_{m=1}^{M}{{\theta _{m}}{z^{(M-m)}}I}\} \\
=\prod\limits_{i=1}^{N}({z^{M+1}}-(1+\theta _{0}-{\alpha \lambda _{i}}){%
z^{M}}-\sum\limits_{m=1}^{M}{{\theta _{m}}{z^{(M-m)}}}).%
\end{array}
\label{W_det}
\end{equation}
Since $\lambda _{1}=0$ and $\sum\limits_{m=0}^{M}{\theta _{m}}=0,$ it is
easy to check that
\begin{equation}
{z^{M+1}}-(1+\theta _{0}-{\alpha \lambda _{1}}){z^{M}}-\sum\limits_{m=1}^{M}{%
{\theta _{m}}{z^{(M-m)}}}=(z-1)h_{1}(z),
\end{equation}%
where $h_{1}(z)$ is given by (\ref{polyh1}). This complete the proof.


\subsection{Proof of Theorem 3}

Note that ${{\theta _{0}+\theta _{1}=0.}}$ It follows from (\ref{polyh1})
and (\ref{polyhiM}) that
\begin{eqnarray}
h_{1}(z) &=&{z}-{{\theta _{0}}},  \label{polyh1M1} \\
h_{2}(z) &=&{z^{2}}-(1+\theta _{0}-{\alpha \lambda _{2}}){z}+{\theta _{0}},
\label{polyh2M1} \\
h_{N}(z) &=&z^{2}-(1+\theta _{0}-{\alpha \lambda _{N}})z+{\theta _{0}}.
\label{polyhNM1}
\end{eqnarray}%
From Theorem 2, the MAS (\ref{M1order-system}) achieves consensus with
convergence rate $\gamma_1$ if and only if the roots of the polynomials $%
h_{i}(z),i=1,2,N$, are within the circle $D(0,\gamma_1),$ where $h_{i}(z)$ is
given by (\ref{polyh1M1})-(\ref{polyhNM1}).

Set $z=\gamma_1\frac{s+1}{s-1}$ with $0 < \gamma_1 \le 1$. Obviously, for $%
i=1,2,N$, $h_{i}(z)=0$ if and only if $f_{i}(s)=0$, where%
\begin{eqnarray}
f_{1}(s) &=&(\gamma_1-\theta _{0})s+\gamma_1+\theta _{0},  \notag \\
f_{2}(s) &=&(\gamma_1^{2}-(1+\theta _{0}-\alpha \lambda _{2})\gamma_1+\theta
_{0})s^{2}+2(\gamma_1^{2}-\theta _{0})s+\gamma_1^{2}+(1+\theta _{0}-\alpha
\lambda _{2})\gamma_1+\theta _{0},  \notag \\
f_{N}(s) &=&(\gamma_1^{2}-(1+\theta _{0}-\alpha \lambda _{N})\gamma_1+\theta
_{0})s^{2}+2(\gamma_1^{2}-\theta _{0})s+\gamma_1^{2}+(1+\theta _{0}-\alpha
\lambda _{N})\gamma_1+\theta _{0}.  \notag
\end{eqnarray}%
Note that $\left\vert z\right\vert \leq \gamma_1$ if and only if $Re(s)\leq 0.$
Hence the roots of $h_{i}(z)$ are within $D(0,\gamma_1)$ if and only if the
roots of $f_{i}(s)$ are in the left-half complex plane (including imaginary
axis), which is equivalent to that all coefficients of $f_{i}(s)$ are
non-negative. Then we get the following optimization problem
\begin{equation}
\begin{array}{l}
\label{optM1}\mathop {\min }\limits_{\{\alpha ,{\theta _{0}}\}}\quad \gamma_1 \\
s.t.\quad \left\{ {%
\begin{array}{l}
{\gamma_1 -{\theta _{0}}\geq 0,} \\
{\gamma_1 +{\theta _{0}}\geq 0,} \\
{{\gamma_1 ^{2}}-(1+{\theta _{0}}-\alpha {\lambda _{2}})\gamma_1 +{\theta _{0}}%
\geq 0,} \\
{{\gamma_1 ^{2}}-(1+{\theta _{0}}-\alpha {\lambda _{N}})\gamma_1 +{\theta _{0}}%
\geq 0,} \\
{{\gamma_1 ^{2}}-{\theta _{0}}\geq 0,} \\
{{\gamma_1 ^{2}}+(1+{\theta _{0}}-\alpha {\lambda _{2}})\gamma_1 +{\theta _{0}}%
\geq 0,} \\
{{\gamma_1 ^{2}}+(1+{\theta _{0}}-\alpha {\lambda _{N}})\gamma_1 +{\theta _{0}}%
\geq 0.}%
\end{array}%
}\right.%
\end{array}%
\end{equation}

Note that $0<\gamma_1\leq 1$ and $\gamma_1=1$ has always feasible solutions. Set
$\gamma_1=1$ to the third inequality, it has $\alpha \lambda _{2}\geq 0.$
Hence $\alpha \geq 0.$ The inequality constraints are then reduced to
\begin{subnumcases}{}
\gamma_1+\theta _{0}\geq 0,  \label{temp1} \\
\gamma_1^{2}-\theta _{0}\geq 0,  \label{temp2} \\
\gamma_1^{2}-(1+\theta _{0}-\alpha \lambda _{2})\gamma_1+\theta _{0}\geq 0,
\label{temp3} \\
\gamma_1^{2}+(1+\theta _{0}-\alpha \lambda _{N})\gamma_1+\theta _{0}\geq 0,
\label{temp4}
\end{subnumcases}


Next we will show that the smallest $\gamma_1$ such that (\ref{temp2})-(\ref%
{temp4}) hold is no less than $\gamma_1^{\ast},$ where $\gamma_1^{\ast}=\frac{%
\sqrt{\lambda _{N}/\lambda _{2}}-1}{\sqrt{\lambda _{N}/\lambda _{2}}+1}.$ By
adding (\ref{temp2}) and (\ref{temp3}), we get $2\gamma_1^{2}-(1+\theta
_{0}-\alpha \lambda _{2})\gamma_1\geq 0,$ hence
\begin{equation}
2\gamma_1-1-\theta _{0}+\alpha \lambda _{2}\geq 0.  \label{temp5}
\end{equation}%
Similarly by adding (\ref{temp2}) and (\ref{temp4}), we have
\begin{equation}
2\gamma_1+1+\theta _{0}-\alpha \lambda _{N}\geq 0.  \label{temp6}
\end{equation}%
By adding (\ref{temp5}) and (\ref{temp6}), we get $\alpha \leq \frac{4\gamma_1%
}{\lambda _{N}-\lambda _{2}}.$ Note that $\theta _{0}\leq \gamma_1^{2},\alpha
\geq 0$ and $0<\gamma_1\leq 1.$ We have%
\begin{eqnarray*}
&&\gamma_1^{2}-(1+\theta _{0}-\alpha \lambda _{2})\gamma_1+\theta _{0} \\
&=&\gamma_1^{2}-\gamma_1+(1-\gamma_1)\theta _{0}+\alpha \lambda _{2}\gamma_1 \\
&\leq &\gamma_1^{2}-\gamma_1+(1-\gamma_1)\gamma_1^{2}+\frac{4\lambda _{2}\gamma_1^{2}}{%
\lambda _{N}-\lambda _{N}}.
\end{eqnarray*}%
It follows from the inequality (\ref{temp3}) that $\gamma_1^{2}-\gamma_1+(1-%
\gamma_1)\gamma_1^{2}+\frac{4\lambda_{2}\gamma_1^{2}}{\lambda _{N}-\lambda _{N}}%
\geq 0.$ This is equivalent to%
\begin{equation}
-\gamma_1^{2}+2\frac{\lambda _{N}+\lambda _{2}}{\lambda _{N}-\lambda _{2}}%
\gamma_1-1\geq 0.  \label{r-equation}
\end{equation}%
The smallest $\gamma_1\in \left( 0,1\right] $ that satisfies (\ref{r-equation}%
) is given by (\ref{r-value}). Hence the solution $\gamma_1$ of (\ref{optM1})
satisfies $\gamma_1\geq \gamma_1^{\ast }.$ By setting $\alpha$ and $\theta _{0}$
as (\ref{alpha-value}) and \ (\ref{theta0-value}) respectively,
we can check that all the inequalities (\ref{temp1})-(\ref{temp4}) are satisfied.
Therefore $\gamma_1^{\ast }$ with parameters $\alpha ^{\ast }$ and $\theta
_{0}^{\ast }$ is the solution of (\ref{optM1}).
This completes the proof.

\subsection{Proof of Theorem 4}

Note that ${{\theta _{0}+\theta _{1}+{\theta _{2}}=0.}}$ It follows from (%
\ref{polyh1}) and (\ref{polyhiM}) that
\begin{eqnarray}
h_{1}(z) &=&{z^{2}}-{{\theta _{0}z}}+{{\theta _{2}}},  \label{polyh1M2} \\
h_{2}(z) &=&{z^{3}}-(1+\theta _{0}-{\alpha \lambda _{2}}){z}^{2}+({{{\theta
_{0}+}}\theta _{2})z-{{\theta _{2}}}},  \label{polyh2M2} \\
h_{N}(z) &=&{z^{3}}-(1+\theta _{0}-{\alpha \lambda _{N}}){z}^{2}+({{{\theta
_{0}+}}\theta _{2})z}-{{{\theta _{2}}}}.  \label{polyhNM2}
\end{eqnarray}%
It follows from Theorem 2 that the MAS (\ref{M2-system}) achieves consensus
with convergence rate $\gamma_2$ if and only if the roots of the polynomials
$h_{i}(z),i=1,2,N$, are within the circle $D(0,\gamma_2),$ where $h_{i}(z)$
is given by (\ref{polyh1M2})-(\ref{polyh2M2}).

Again, set $z=\gamma_2\frac{s+1}{s-1}$ with $0 < \gamma_2 \le 1$.
Then, for $i=1,2,N$, $h_{i}(z)=0$ if and only if $f_{i}(s)=0$, where%
\begin{eqnarray*}
f_{1}(s) &=&(\gamma_2^{2}-\theta _{0}\gamma_2+\theta
_{2})s^{2}+2(\gamma_2^{2}-\theta _{2})s+\gamma_2^{2}+\theta _{0}\gamma_2+\theta
_{2}, \\
f_{2}(s) &=&(\gamma_2^{3}-(1+\theta _{0}-{\alpha }\lambda
_{2})\gamma_2^{2}+(\theta _{0}+\theta _{2})\gamma_2-\theta
_{2})s^{3}+(3\gamma_2^{3}-(1+\theta _{0}-{\alpha }\lambda _{2})
\gamma_2^{2}-(\theta_{0}+\theta _{2})\gamma_2 \\
&&+3\theta _{2})s^{2}+(3\gamma_2^{3}+(1+\theta _{0}-{\alpha }\lambda
_{2})\gamma_2^{2}-(\theta _{0}+\theta _{2})\gamma_2-3\theta
_{2})s+\gamma_2^{3}+(1+\theta _{0}-{\alpha }\lambda _{2})\gamma_2^{2} \\
&&+(\theta_{0}+\theta _{2})\gamma_2+\theta _{2}, \\
f_{N}(s) &=&(\gamma_2^{3}-(1+\theta _{0}-{\alpha }\lambda
_{N})\gamma_2^{2}+(\theta _{0}+\theta _{2})\gamma_2-\theta
_{2})s^{3}+(3\gamma_2^{3}-(1+\theta _{0}-{\alpha }\lambda _{N})\gamma_2^{2}
-(\theta _{0}+\theta_{2})\gamma_2 \\
&&+3\theta _{2})s^{2}+(3\gamma_2^{3}+(1+\theta _{0}-{\alpha }\lambda
_{N})\gamma_2^{2}-(\theta _{0}+\theta _{2})\gamma_2-3\theta
_{2})s+\gamma_2^{3}+(1+\theta _{0}-{\alpha }\lambda _{N})\gamma_2^{2} \\
&&+(\theta _{0}+\theta_{2})\gamma_2+\theta _{2}.
\end{eqnarray*}

\begin{table*}[htbp]
\centering
  \caption{The Routh table of $f_2(s)=0$}
\begin{tabular}{c|cc}
\hline
$s^3$              & $\gamma_2^3-(1+\theta_0-\alpha\lambda_2)\gamma_2^2+(\theta_0+\theta_2)
\gamma_2-\theta_2$          & $3\gamma_2^3+(1+\theta_0-\alpha\lambda_2)\gamma_2^2-(\theta_0+
\theta_2)\gamma_2-3\theta_2$          \\
$s^2$              & $3\gamma_2^3-(1+\theta_0-\alpha\lambda_2)\gamma_2^2-(\theta_0+\theta_2)
\gamma_2-3\theta_2$          & $\gamma_2^3+(1+\theta_0-\alpha\lambda_2)\gamma_2^2+(\theta_0
+\theta_2)\gamma_2+\theta_2$          \\
$s^1$              & $\frac{8(\gamma_2^6-(\theta_0+\theta_2)\gamma_2^4+(1+\theta_0-\alpha
\lambda_2)\theta_2\gamma_2^2-\theta_2^2)}{3\gamma_2^3-(1+\theta_0-\alpha\lambda_2)
\gamma_2^2-(\theta_0+\theta_2)\gamma_2-3\theta_2}$          &           \\
$s^0$              & $\gamma_2^3+(1+\theta_0-\alpha\lambda_2)\gamma_2^2+(\theta_0
+\theta_2)\gamma_2+\theta_2$          &           \\
\hline
\end{tabular}
\end{table*}

\begin{table*}[htbp]
\centering
  \caption{The Routh table of $f_N(s)=0$}
\begin{tabular}{c|cc}
\hline
$s^3$              & $\gamma_2^3-(1+\theta_0-\alpha\lambda_N)\gamma_2^2+(\theta_0+\theta_2)
\gamma_2-\theta_2$          & $3\gamma_2^3+(1+\theta_0-\alpha\lambda_N)\gamma_2^2-(\theta_0+
\theta_2)\gamma_2-3\theta_2$          \\
$s^2$              & $3\gamma_2^3-(1+\theta_0-\alpha\lambda_N)\gamma_2^2-(\theta_0+\theta_2)
\gamma_2-3\theta_2$          & $\gamma_2^3+(1+\theta_0-\alpha\lambda_N)\gamma_2^2+(\theta_0
+\theta_2)\gamma_2+\theta_2$          \\
$s^1$              & $\frac{8(\gamma_2^6-(\theta_0+\theta_2)\gamma_2^4+(1+\theta_0-\alpha
\lambda_N)\theta_2\gamma_2^2-\theta_2^2)}{3\gamma_2^3-(1+\theta_0-\alpha\lambda_N)
\gamma_2^2-(\theta_0+\theta_2)\gamma_2-3\theta_2}$          &           \\
$s^0$              & $\gamma_2^3+(1+\theta_0-\alpha\lambda_N)\gamma_2^2+(\theta_0
+\theta_2)\gamma_2+\theta_2$          &           \\
\hline
\end{tabular}
\end{table*}

Note that $\left\vert z\right\vert \leq \gamma_2$ if and only if $Re(s)\leq 0.$
Construct the Routh table corresponding to $f_2(s)=0$ and $f_N(s)=0$
shown in Table III and IV, respectively. Based on the Routh stability criterion,
the optimal convergence rate can be
obtained by solving the following optimization problem
\begin{equation}
\begin{array}{l}
\label{optM2}\mathop {\min }\limits_{\{\alpha ,{\theta _{0}},{\theta _{2}}\}%
}\quad \gamma_2 \\
s.t.\quad \left\{ {%
\begin{array}{l}
{{\gamma_2 ^{3}}-(1+\theta _{0}-{\alpha }\lambda _{2}){\gamma_2 ^{2}}+(\theta
_{0}+\theta _{2})\gamma_2 -{\theta _{2}}\geq 0}, \\
{3{\gamma_2 ^{3}}-(1+\theta _{0}-{\alpha }\lambda _{2}){\gamma_2 ^{2}}-(\theta
_{0}+\theta _{2})\gamma_2 +3{\theta _{2}}\geq 0}, \\
{3{\gamma_2 ^{3}}+(1+\theta _{0}-{\alpha }\lambda _{N}){\gamma_2 ^{2}}-(\theta
_{0}+\theta _{2})\gamma_2 -3{\theta _{2}}\geq 0}, \\
{{\gamma_2 ^{3}}+(1+\theta _{0}-{\alpha }\lambda _{N}){\gamma_2 ^{2}}+(\theta
_{0}+\theta _{2})\gamma_2 +{\theta _{2}}\geq 0}, \\
{{\gamma_2 ^{2}}-{\theta _{0}}\gamma_2 +{\theta _{2}}\geq 0,} \\
{2{\gamma_2 ^{2}}-2{\theta _{2}}\geq 0}, \\
{{\gamma_2 ^{2}}+{\theta _{0}}\gamma_2 +{\theta _{2}}\geq 0}, \\
{{\gamma_2 ^{6}}-(\theta _{0}+\theta _{2}){\gamma_2 ^{4}}+(1+\theta _{0}-{\alpha
}\lambda _{2}){\theta _{2}}{\gamma_2 ^{2}}-{\theta _{2}}^{2}\geq 0}, \\
{{\gamma_2 ^{6}}-(\theta _{0}+\theta _{2}){\gamma_2 ^{4}}+(1+\theta _{0}-{\alpha
}\lambda _{N}){\theta _{2}}{\gamma_2 ^{2}}-{\theta _{2}}^{2}\geq 0}, \\
{{\gamma_2 ^{3}}+(1+\theta _{0}-{\alpha }\lambda _{2}){\gamma_2 ^{2}}+(\theta
_{0}+\theta _{2})\gamma_2 +{\theta _{2}\geq 0}}, \\
{3{\gamma_2 ^{3}}+(1+\theta _{0}-{\alpha }\lambda _{2}){\gamma_2 ^{2}}-(\theta
_{0}+\theta _{2})\gamma_2 -3{\theta _{2}\geq 0}}, \\
{{\gamma_2 ^{3}}-(1+\theta _{0}-{\alpha }\lambda _{N}){\gamma_2 ^{2}}+(\theta
_{0}+\theta _{2})\gamma_2 -{\theta _{2}\geq 0}}, \\
{3{\gamma_2 ^{3}}-(1+\theta _{0}-{\alpha }\lambda _{N}){\gamma_2 ^{2}}-(\theta
_{0}+\theta _{2})\gamma_2 +3{\theta _{2}\geq 0}}.%
\end{array}%
}\right.%
\end{array}%
\end{equation}
From the first inequality in (\ref{optM2}), we have 
${{\gamma_2^{3}}-(1+\theta _{0}-{\alpha }\lambda _{2}){\gamma_2^{2}}+(\theta
_{0}+\theta _{2})\gamma_2-{\theta _{2}}} 
={\alpha }\lambda _{2}{\gamma_2}^{2}-{(1-\gamma_2)(\gamma_2}^{2}{-{\theta _{0}}%
\gamma_2+{\theta _{2})\geq 0}}$. 
Since $\gamma_2\in \left( 0,1\right] ,$ then it follows that ${{{\alpha }%
\lambda _{2}\gamma_2}}^{2}\geq 0.$ Hence ${\alpha \geq 0.}$ Then the
inequalities constraints can be further reduced to
\begin{subnumcases}{}
{{\gamma_2^{3}}-(1+\theta _{0}-{\alpha }\lambda _{2}){\gamma_2^{2}}+(\theta _{0}+\theta
_{2})\gamma_2-{\theta _{2}}\geq 0},  \label{M2-tmp4} \\
{3{\gamma_2^{3}}-(1+\theta _{0}-{\alpha }\lambda _{2}){\gamma_2^{2}}-(\theta _{0}+\theta
_{2})\gamma_2+3{\theta _{2}}\geq 0},  \label{M2-tmp5} \\
{3{\gamma_2^{3}}+(1+\theta _{0}-{\alpha }\lambda _{N}){\gamma_2^{2}}-(\theta _{0}+\theta
_{2})\gamma_2-3{\theta _{2}}\geq 0},  \label{M2-tmp6} \\
{{\gamma_2^{3}}+(1+\theta _{0}-{\alpha }\lambda _{N}){\gamma_2^{2}}+(\theta _{0}+\theta
_{2})\gamma_2+{\theta _{2}}\geq 0},  \label{M2-tmp7} \\
{{\gamma_2^{2}}-{\theta _{0}}\gamma_2+{\theta _{2}}\geq 0,}  \label{M2-tmp1} \\
{2{\gamma_2^{2}}-2{\theta _{2}}\geq 0},  \label{M2-tmp2} \\
{{\gamma_2^{2}}+{\theta _{0}}\gamma_2+{\theta _{2}}\geq 0},  \label{M2-tmp3} \\
{{\gamma_2^{6}}-(\theta _{0}+\theta _{2}){\gamma_2^{4}}+(1+\theta _{0}-{\alpha }\lambda
_{2}){\theta _{2}}{\gamma_2^{2}}-{\theta _{2}}^{2}\geq 0},  \label{M2-tmp8} \\
{{\gamma_2^{6}}-(\theta _{0}+\theta _{2}){\gamma_2^{4}}+(1+\theta _{0}-{\alpha }\lambda
_{N}){\theta _{2}}{\gamma_2^{2}}-{\theta _{2}}^{2}\geq 0}.  \label{M2-tmp9}
\end{subnumcases}
Next, we will show that the optimization problem $\min_{\alpha ,\theta
_{0},\theta _{2}}{\gamma_2}$ with constraints (\ref{M2-tmp4})-(\ref{M2-tmp7})
has a unique solution as (\ref{r-valueM2}) with the corresponding parameters
(\ref{alphaM2})-(\ref{theta2M2}).

By adding (\ref{M2-tmp4}) and (\ref{M2-tmp6}), we get
\begin{equation}
-4\theta_2-(\lambda_N-\lambda_2){\gamma_2}^2\alpha+4{\gamma_2}^3 \geq 0.
\label{M2-temp14}
\end{equation}%
Similarly by adding (\ref{M2-tmp5}) and (\ref{M2-tmp7}), we have
\begin{equation}
4\theta_2-(\lambda_N-\lambda_2){\gamma_2}^2\alpha+4{\gamma_2}^3 \geq 0.
\label{M2-temp15}
\end{equation}%
By adding (\ref{M2-temp14}) and (\ref{M2-temp15}), we get $%
-2(\lambda_N-\lambda_2){\gamma_2}^{2}\alpha+8{\gamma_2}^{3} \geq 0$, hence
\begin{equation}
(\lambda_N-\lambda_2)\alpha+4{\gamma_2} \geq 0.  \label{M2-temp16}
\end{equation}%
Multiplying $\frac{1+\gamma_2}{1-\gamma_2}$ to (\ref{M2-tmp4}) and adding (\ref%
{M2-tmp5}), we have
\begin{equation}
(1-\gamma_2)\theta_2+\frac{\lambda_2{\gamma_2}^{2}}{1-\gamma_2}\alpha-{\gamma_2}%
^2(1-\gamma_2) \geq 0.  \label{M2-temp17}
\end{equation}%
Multiplying $\frac{4}{1-\gamma_2}$ to (\ref{M2-temp17}) and adding (\ref{M2-temp14}%
), we get
\begin{equation}
(\frac{4\lambda_2}{(1-\gamma_2)^2}-(\lambda_N-\lambda_2))\alpha-4(1-\gamma_2)
\geq 0.  \label{M2-temp18}
\end{equation}%
It follows from (\ref{M2-temp16}) and (\ref{M2-temp18}) that
\begin{equation}
\frac{4\gamma_2}{\lambda _{N}-\lambda _{2}}\geq \alpha \geq \frac{4(1-\gamma_2)}{%
\frac{4\lambda _{2}}{(1-\gamma_2)^{2}}-(\lambda _{N}-\lambda _{2})}.
\label{Ineq-key}
\end{equation}%
The smallest $\gamma_2$ for (\ref{Ineq-key}) to hold must satisfy%
\begin{equation*}
\frac{4\gamma_2}{\lambda _{N}-\lambda _{2}}=\frac{4(1-\gamma_2)}{\frac{4\lambda
_{2}}{(1-\gamma_2)^{2}}-(\lambda _{N}-\lambda _{2})}.
\end{equation*}%
This is equivalent to
\begin{equation*}
\frac{4\lambda _{2}\gamma_2}{(1-\gamma_2)^{2}}-(\lambda _{N}-\lambda
_{2})\gamma_2=(\lambda _{N}-\lambda _{2})(1-\gamma_2),
\end{equation*}%
\begin{equation*}
4\lambda _{2}\gamma_2=(\lambda _{N}-\lambda _{2})(1-\gamma_2)^{2},
\end{equation*}%
and $\gamma_2^{2}-2\frac{\lambda _{N}+\lambda _{2}}{\lambda _{N}-\lambda
_{2}}\gamma_2+1=0.$ The solution of $\gamma_2$ from the last equation is (\ref{r-valueM2}) since $\gamma_2\in \left(
0,1\right] .$ Substitute (\ref{r-valueM2}) into (\ref{Ineq-key}), we have (\ref%
{alphaM2}).

By substituting $\gamma_2^{\ast },\alpha ^{\ast },\theta _{0}^{\ast }$ and $%
\theta _{2}^{\ast }$ to the inequalities (\ref{M2-tmp1})-(\ref{M2-tmp9}), we
can check that all of them hold. This completes the proof.

\subsection{Proof of Lemma 5 }

Let $z=\frac{s+1}{s-1}.$ The corresponding continuous system of $P(rz;%
\underline{\lambda })$ is
\begin{equation*}
P(s)=\frac{\underline{\lambda }}{1-r}\frac{s-1}{-s+\frac{1+r}{1-r}}.
\end{equation*}%
Next we use Lemma 2 to compute the optimal gain margin. Do coprime
factorization of $P(s)=\frac{U(s)}{V(s)}$ as follows
\begin{eqnarray*}
U(s) &=&\frac{\underline{\lambda }}{1-r}\frac{s-1}{s+1}, \\
V(s) &=&\frac{-s+\frac{1+r}{1-r}}{s+1}, \\
Y_{u}(s) &=&\frac{1-r}{\underline{\lambda }r}\text{ and }Y_{v}(s)=\frac{1-r}{r}.
\end{eqnarray*}
It is easy to check that $U(s)Y_{u}(s)+V(s)Y_{v}(s)=1$ and $U(s)V(s)$
has two zeros, $c_{1}=1$ and $c_{2}=\frac{%
1+r}{1-r}$, in the right half plane. Then we have
\begin{eqnarray*}
b_{1} &=&U(c_{1})Y_{u}(c_{1})=0, \\
b_{2} &=&U(c_{2})Y_{u}(c_{2})=\frac{\frac{1+r}{1-r}-1}{r(\frac{1+r}{1-r}+1)}%
=1.
\end{eqnarray*}%
It follows from direct computation that
\begin{equation*}
B_{1}=\left[
\begin{array}{cc}
\frac{1}{2} & \frac{1-r}{2} \\
\frac{1-r}{2} & \frac{1-r}{2(1+r)}%
\end{array}%
\right] ,B_{2}=\left[
\begin{array}{cc}
0 & 0 \\
0 & \frac{1-r}{2(1+r)}%
\end{array}%
\right] .
\end{equation*}%
Then we have $\gamma _{\inf }=\sqrt{\overline{\lambda }(B_{1}^{-1}B_{2})}=%
\frac{1}{r}$ and $k_{\sup }=\left( \frac{\gamma _{\inf }+1}{\gamma _{\inf }-1%
}\right) ^{2}=\left( \frac{r+1}{r-1}\right) ^{2}.$ This completes the proof.

\end{document}